\pdfoutput=1

\documentclass[11pt]{amsart}
\usepackage{amsmath,amsthm,amssymb,amscd,eucal}        
\usepackage{graphicx, color}
\usepackage[all]{xy}
\usepackage[margin=1.3in]{geometry}
\usepackage{enumerate}

\begin{document}

\definecolor{mains}{cmyk}{0, .80, .80, 0} 

\newcommand{\half}{{\textstyle{\frac{1}{2}}}}
\newcommand{\balpha}{{\boldsymbol{\alpha}}}
\newcommand{\B}{{\bf 1}}
\newcommand{\e}{{\bf e}}
\newcommand{\y}{{\bf r}}
\newcommand{\x}{{\bf x}} 
\newcommand{\tx}{{\bar{x}}}                   
\newcommand{\bn}{{\bf n}}

%%%%%%%%%%%%%%%%%%%%%%%%%%%%%%%%%%%%%%%%%%%%%%%%%%%%%%%%%%%%%%%%%%%%%%%%%%%%%%%%%%%%%%%%%
%
%  notation
%
%%%%%%%%%%%%%%%%%%%%%%%%%%%%%%%%%%%%%%%%%%%%%%%%%%%%%%%%%%%%%%%%%%%%%%%%%%%%%%%%%%%%%%%%%

\newcommand{\aff}  {{\rm Aff}_+}                 % Affine aut(+ dilation) 
\newcommand{\Conf}{{\rm Config}}                 %Configuration space
\newcommand{\conf}{{\rm config}}                 %Configuration space
 
\newcommand{\PSL} {\Pj\Sl_2(\R)}                           %PSL_2(R)
\newcommand{\PGL} {\Pj\Gl_2(\R)}                           %PGL_2(R)
\newcommand{\PGLC} {\Pj\Gl_2(\C)}                          %PGL_2(C)
\newcommand{\RP} {\R\Pj^1}                                 %RP^1
\newcommand{\CP} {\C\Pj^1}                                 %CP^1
\newcommand{\C} {{\mathbb C}}                              %complex
\newcommand{\R} {{\mathbb R}}                              %reals
\newcommand{\Z} {{\mathbb Z}}                              %integers
\newcommand{\Pj} {{\mathbb P}}                             %projective space
\newcommand{\Sg} {{\mathbb S}}                             %symmetric group
\newcommand{\Gl} {{\rm Gl}}                                %g.linear group
\newcommand{\Sl} {{\rm Sl}}                                %s.linear group
\newcommand{\SO} {{\rm SO}}                                %s.ortho group
\newcommand{\orb}{{\rm orb}}
\newcommand{\PGl}{{\rm PGl}} 
\newcommand{\SP}{{\rm SP}}
\newcommand{\Ot}{{\rm O}}
\newcommand{\D}{{\rm D}}
\newcommand{\reg}{{\rm reg}}

\newcommand{\cA}{{\mathcal A}}
\newcommand{\cB}{{\mathcal B}}
\newcommand{\cC}{{\mathcal C}}
\newcommand{\blambda}{{\boldsymbol{\lambda}}}
\newcommand{\bDelta}{{\boldsymbol{\Delta}}}
\newcommand{\bdelta}{{\boldsymbol{\delta}}}
\newcommand{\cQ}{{\mathcal Q}}
\newcommand{\dQ}{{\mathcal Q}^0}                               %[!CHANGED!]
\newcommand{\qQ}{{\widehat{\cQ}}}
\newcommand{\tQ}{\widetilde{\cQ}}
\newcommand{\sQ}{{\sf Q}}

\newcommand{\cell}{{\rm Cell}}
\newcommand{\cells}{{\rm Cells}}
\newcommand{\iso}{{\rm Iso}}
\newcommand{\metmap}{\mathfrak{d}}
\newcommand{\dis} {{\mathcal D}}                                 %discriminant
\newcommand{\diagonal}{\triangle}                                %thick diagonal

\newcommand{\Sp}{\mathcal S}

\newcommand{\sn}[1]{{\rm CSN}_{#1}}
\newcommand{\n}[1]{\mathfrak{S}_{#1}}
\newcommand{\mc}{\mathcal}
\newcommand{\ms}{\mathcal}
\newcommand{\bd}{overlap space }

\newcommand{\oM} [1] {\ensuremath{{\mathcal M}_{0,#1}(\R)}}               %open real moduli
\newcommand{\M} [1] {\ensuremath{{\overline{\mathcal M}}{_{0, #1}(\R)}}}  %real moduli
\newcommand{\OM} [1] {\ensuremath{{\overline{\mathcal M}}{_{0, #1}^{{\rm or}}(\R)}}}
\newcommand{\KM} [1] {\ensuremath{{\overline{\mathcal M}}{_{0, #1}^{{\rm kap}}(\R)}}}
\newcommand{\cM} [1] {\ensuremath{{\mathcal M}_{0, #1}}}                  %open complex
\newcommand{\CM} [1] {\ensuremath{{\overline{\mathcal M}}_{0, #1}}}       %compact complex

\newcommand{\sM} [1] {\ensuremath{{\sf M}_{0,#1}(\R)}}
\newcommand{\tsM} [1]{\ensuremath{{\widetilde{\sf M}}{_{0, #1}(\R)}}}
\newcommand{\tsQ} [1] {{\widetilde{\sf Q}}}         %new (cf \S 8)

\newcommand{\iDelta}{{\overset{o}\Delta}}

\newcommand{\bT}{{\overline{\mathcal T}}}
\newcommand{\T}[1]{{\mathcal T}_#1}                                         % Tree space
\newcommand{\pT}[1]{{\bf T}_{#1}(\R)}             
\newcommand{\BHV}[1]{{\rm BHV}_{#1}}                    % BHV space of trees
\newcommand{\iBHV}[1]{{\rm BHV}_{#1}^+}                 % BHV space of trees [!CHANGED!]

\newcommand{\hide}[1]{}

\newcommand{\suchthat} {\:\: | \:\:}
\newcommand{\ore} {\ \ {\it or} \ \ }
\newcommand{\oand} {\ \ {\it and} \ \ }

%%%%%%%%%%%%%%%%%%%%%%%%%%%%%%%%%%%%%%%%%%%%%%%%%%%%%%%%%%%%%%%%%%%%%%%%%%%%%%%%%%%%%%%%%
%
%  paper formatting
%
%%%%%%%%%%%%%%%%%%%%%%%%%%%%%%%%%%%%%%%%%%%%%%%%%%%%%%%%%%%%%%%%%%%%%%%%%%%%%%%%%%%%%%%%%

\theoremstyle{plain}
\newtheorem{thm}{Theorem}
\newtheorem{prop}[thm]{Proposition}
\newtheorem{cor}[thm]{Corollary}
\newtheorem{lem}[thm]{Lemma}
\newtheorem{conj}[thm]{Conjecture}
\newtheorem*{thmhull}{Theorem 11}
\theoremstyle{definition}
\newtheorem*{defn}{Definition}
\newtheorem*{exmp}{Example}
\newtheorem*{claim}{Claim}

\theoremstyle{remark}
\newtheorem*{rem}{Remark}
\newtheorem*{hnote}{Historical Note}
\newtheorem*{nota}{Notation}
\newtheorem*{ack}{Acknowledgments}
\numberwithin{equation}{section}

%%%%%%%%%%%%%%%%%%%%%%%%%%%%%%%%%%%%%%%%%%%%%%%%%%%%%%%%%%%%%%%%%%%%%%%%%%%%%%%%%%%%%%%%%

\title{A space of phylogenetic networks}

\author{Satyan L.\ Devadoss}
\address{S.\ Devadoss: University of San Diego, San Diego, CA 92110}
\email{devadoss@sandiego.edu}

\author{Samantha Petti}
\address{S.\ Petti: Georgia Tech, Atlanta, GA 30332}
\email{spetti@gatech.edu}

\begin{abstract}
A classic problem in computational biology is constructing a phylogenetic tree given a set of distances between $n$ species.  In most cases, a tree structure is too constraining.  We consider a circular split network, a generalization of a tree in which multiple parallel edges signify divergence. A geometric space of such networks is introduced, forming a natural extension of the work by Billera, Holmes, and Vogtmann on tree space.  We explore properties of this space, and show a natural embedding of the compactification of the real moduli space of curves within it.
\end{abstract}

\subjclass[2010]{52B11, 14H10, 92B10, 05E45}

\keywords{phylogenetics, split networks, associahedron, moduli space}

\maketitle

\baselineskip=17pt

%%%%%%%%%%%%%%%%%%%%%%%%%%%%%%%%%%%%%%%%%%%%%%%%%%%%%%%%%%%%%%%%%%%%%%%%%%%%%%%%%%%%%%%%%
%                                                      
%                Introduction
%
%%%%%%%%%%%%%%%%%%%%%%%%%%%%%%%%%%%%%%%%%%%%%%%%%%%%%%%%%%%%%%%%%%%%%%%%%%%%%%%%%%%%%%%%%
\section{Introduction} 
\label{s:intro}

A classical problem in computational biology is the construction of a phylogenetic tree from a sequence alignment of $n$ species. The main tool used to build a phylogenetic tree from this data involves computing the maximum likelihood estimate (MLE) for each of the $(2n-5)!!$ possible trees with $n$ leaves, where each leaf corresponds to a given species.  Such a procedure is quite difficult and requires examining an exponential number of trees.  One method to circumvent this problem is distance based:  One can construct a distance between two species (such as Hamming distance) which records the proportion of characters where the two species differ (based on certain genetic characteristics).  Such a record can be encoded by an $n \times n$ real symmetric, nonnegative matrix called a \emph{dissimilarity matrix}.  The phylogenetic problem is then to reconstruct a weighted tree (on edges) that represents this matrix.

In order to analyze tree-like data, it is necessary not just to understand individual tree structures, but also the relationships between them.  Billera, Holmes, and Vogtmann laid the foundation for this process by constructing a space $\BHV{n}$ of such metric trees with nonpositive curvature, making it useful for geometric methods, such as the calculation of geodesics and centroids \cite{owe}.  

Unfortunately, most dissimilarity matrices are not tree metrics, as data from sequence alignments often fails due to hybridization, horizontal gene transfer, recombination, or gene duplication \cite{hrs}. However, under a weaker (Kalmanson) condition, one obtains a \emph{circular split network} instead of a tree as output, a generalization of a tree in which multiple parallel edges signify divergence. In this work, we construct a geometric space $\sn{n}$ for such networks, similar to $\BHV{n}$, and explore properties of this space.  Interestingly, this space appears in the work of Hacking, Keel, and Tevelev \cite[Section 8]{hkt} based on the real algebraic geometry of del Pezzo surfaces and the root system $D_n$.
Devadoss and Morava \cite{dm} constructed a smooth blowup 
$$\ real(\CM{n}) \ \longrightarrow \ trop(\CM{n}) \ \simeq \ \BHV{n} $$
between the real points and the tropicalization \cite{ss} of the compactified moduli space of stable genus zero algebraic curves marked with distinct smooth points.  In this paper, we show an embedding 
$$\ real(\CM{n}) \ \hookrightarrow \ \sn{n} \, ,$$
in Theorem~\ref{t:moduliembed} which endows this real moduli space with an inherent metric.  We claim that $\sn{n}$ seems to be a canonical space of study, with natural embeddings of both $\BHV{n}$ and \M{n} within its structure.

Section~\ref{s:defns} provides foundational definitions of trees, splits, and dissimilarity matrices.  The construction of the space of split networks, along with understanding its properties, is given in Section~\ref{s:space} and \ref{s:props}.  The associahedron polytope and the real moduli space of curves is introduced in Section~\ref{s:moduli}, along with the embedding map provided in Section~\ref{s:embed}.

\begin{ack}
We would like to thank Jack Morava, Lior Pachter, and Jim Stasheff for their encouragement, along with helpful conversations with Sean Keel and Megan Owen.  Devadoss was partially supported by a John Templeton Foundation grant 51894.  We also thank Williams College, where this work was mostly completed during summer 2015.
\end{ack} 

%%%%%%%%%%%%%%%%%%%%%%%%%%%%%%%%%%%%%%%%%%%%%%%%%%%%%%%%%%%%%%%%%%%%%%%%%%%%%%%%%%%%%%%%%
%
%                Splits, Trees, and Networks
%
%%%%%%%%%%%%%%%%%%%%%%%%%%%%%%%%%%%%%%%%%%%%%%%%%%%%%%%%%%%%%%%%%%%%%%%%%%%%%%%%%%%%%%%%%
\section{Splits, Trees, and Networks} 
\label{s:defns}
\subsection{}

We begin with some foundational definitions; for a more thorough introduction to phylogenetic trees and networks, we refer the reader to \cite{hrs, lp, sest}.  Throughout the paper, $X$ denotes the finite set of $n$ distinct species on which metrics are defined.

\begin{defn} 
A \emph{split} $S=\{A,B\}$ is a partition of $X$ into two nonempty sets.  A split is \textit{trivial} if one set of the partition has cardinality one.  A set of splits is called a \textit{split system}. 
\end{defn}

\begin{defn} 
A split system $\Sp$ is \emph{pairwise compatible} if for every pair of splits $S_1=\{A_1, B_1\}$ and $S_2=\{A_2,B_2\}$ in $\Sp$, at least one of the following is empty:
$$A_1 \cap A_2, \ \ A_1 \cap B_2, \ \ A_2 \cap B_1, \ \ B_1 \cap B_2.$$
\end{defn}

\noindent
There is a canonical graph associated with a split system, called the \emph{Buneman graph} \cite{sest}.  When the split system is pairwise compatible, this graph is a tree, where the vertices are the species $X$, and the edges are the splits.  Figure~\ref{f:split} displays corresponding tree for each set of pairwise compatible splits, where each edge visually represents a split.  Notice that the order in which  splits are introduced is independent of the associated Buneman tree.

\begin{figure}[h]
\includegraphics{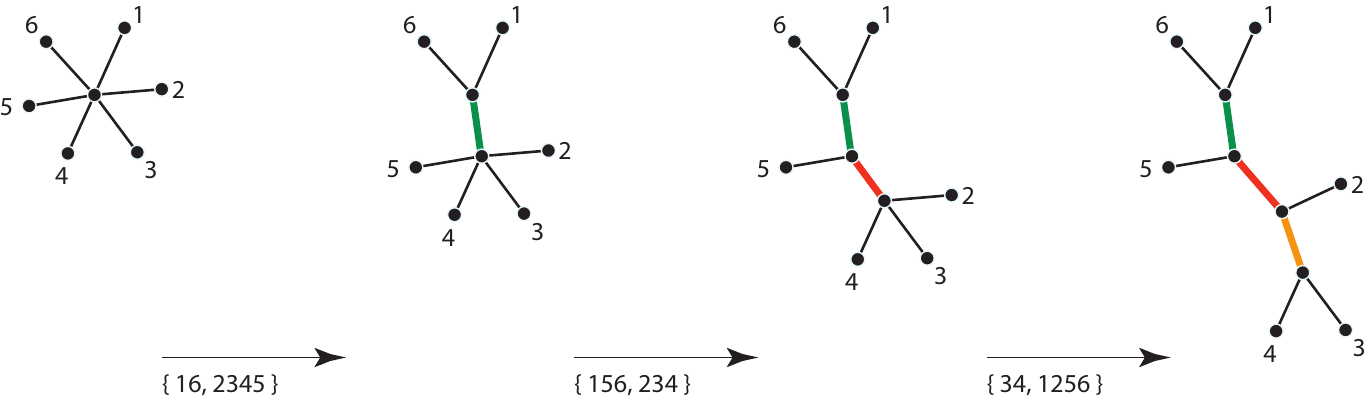}
\caption{Pairwise compatible splits and the associated Buneman trees.}
\label{f:split}
\end{figure}

Thus far, we have been considering combinatorial and topological properties. However, biological data often comes in the form of distances between the species of $X$ rather than in terms of splits, which can be encoded by an $n \times n$ real symmetric, nonnegative map 
$$\delta : X \times X \longrightarrow \R$$ 
called a \emph{dissimilarity matrix}.  The phylogenetic problem is then to reconstruct an edge-weighted tree where $\delta(i,j)$ is the additive distance between species $i$ and $j$, provided such a tree exists.  The following gives necessary and sufficient conditions for such an existence.

\begin{thm} \cite[Chapter 7]{sest} \label{t:4pt}
The dissimilarity matrix $\delta$ realizes a weighted tree if and only if it satisfies the \emph{four-point condition}: if for every four elements $i,j,k,l \in X$, two of the three terms below are greater than or equal to the third:
$$\delta(i,j) +\delta(k,l), \ \ \ \delta(i,l) +\delta(j,k), \ \ \ \delta(i,k) +\delta(j,l)\,.$$
\end{thm}

Due to this result, a matrix satisfying the four-point condition is called a \emph{tree metric}.
The \emph{neighbor-joining} algorithm \cite{sn} explicitly constructs the weighted tree from a dissimilarity matrix satisfying the four-point condition, viewed as a map between such matrices to pairwise compatible split systems.

%%%%%%%%%%%%%%%%%%%%%%%%%%%%%%%%%%%%%%%%%%%%%%%%%%%%%%%%%%%%%%%%%%%%%%%%%%%%%%%%%%%%%%%%%
\subsection{}

Most data from sequence alignments often fails to satisfy the four-point condition due to hybridization, horizontal gene transfer, recombination, or gene duplication.  One can relax conditions to obtain a \emph{split network}, a generalization of a tree where each split is represented by a set of parallel edges that disconnects the graph according to the partition.  Figure~\ref{f:network} displays an example of the construction of a split network, viewed as ``pulling" each additional split in a different direction (and shown using distinct colors).  Geometrically, the lengths of each parallel set of edges are identical.

\begin{figure}[h]
\includegraphics{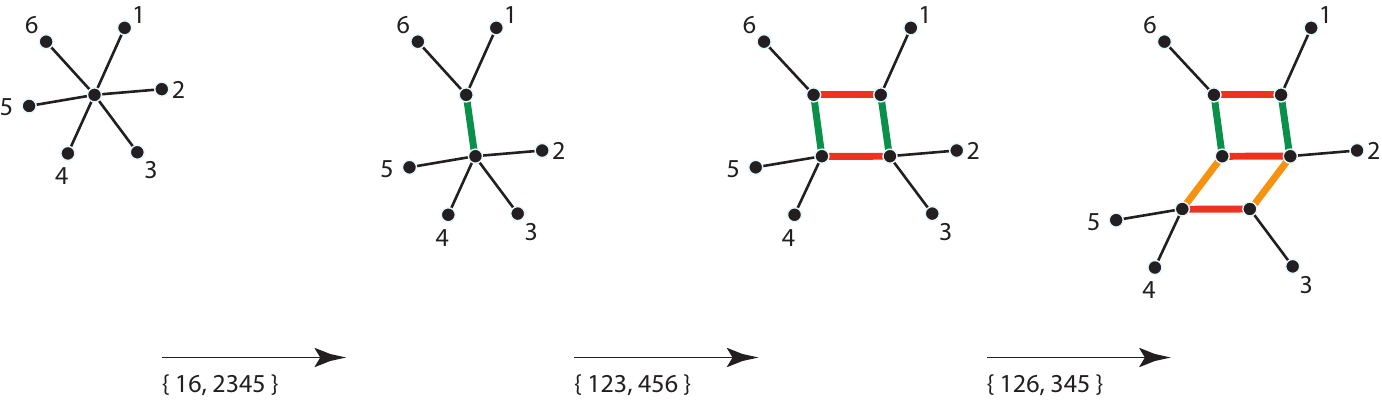}
\caption{Splits and the associated split network.}
\label{f:network}
\end{figure}

The \emph{convex hull} algorithm can be used to draw any split system. It works by increasing the dimensionality to add splits to a growing split network \cite{hrs}.  In general, for a system with $k$ splits, the visualization of the network leads to the 1-skeleton of the $k$-cube.  

Although any split system can be represented as a split network, not all are planar in nature.  For instance, adding the split $\{14, 2356\}$ to the network of Figure~\ref{f:network} creates a network in 3D, the additional dimension separating $\{14\}$ from $\{2356\}$.  To preserve planarity, the crucial concept needed is that of a \emph{circular ordering} for the species set X.

\begin{defn} 
A split system $\Sp$ is \emph{circular} with respect to some cyclic ordering $\pi = \{x_1, \dots x_n\}$ of $X$ if every split of $\Sp$ is of the form $\big\{\{x_{i+1}, \dots, x_j\}, \{x_{j+1}, \dots, x_i\}\big\}$, where $x_{n+1}=x_1$.
\end{defn}

The following establishes necessary and sufficient conditions for determining when a dissimilarity map can be realized as an edge-weighted circular split network, analogous to the four-point condition for trees of Theorem~\ref{t:4pt} above.

\begin{thm} \cite{bd} \label{t:kalmanson}
The dissimilarity matrix $\delta$ realizes a weighted circular split system if and only if it satisfies the \emph{Kalmanson condition} with respect to some circular ordering $\pi$:  if for every $i<j<k<l$ in the circular ordering $\pi$, both inequalities hold:  
$$ \delta(x_i,x_j) +\delta(x_k,x_l) \leq \delta(x_i,x_k) + \delta(x_j,x_l)$$ 
$$ \delta(x_i,x_l) +\delta(x_j,x_k) \leq \delta(x_i,x_k) + \delta(x_j,x_l).$$ 
\end{thm}

Due to this result, a matrix satisfying the Kalmanson condition is called a \emph{circular decomposable  metric}.  Such metrics can be realized by a circular split network in which the shortest path between any two leaves contains precisely one edge from each split that separates the two leaves. Thus, the length of this path is the distance prescribed by the metric.  Given a dissimilarity matrix satisfying the Kalmanson condition, the \emph{neighbor-net} algorithm \cite{bm} produces a circular ordering and a weighted split network compatible with that ordering.

%%%%%%%%%%%%%%%%%%%%%%%%%%%%%%%%%%%%%%%%%%%%%%%%%%%%%%%%%%%%%%%%%%%%%%%%%%%%%%%%%%%%%%%%%
\subsection{}

Unlike trees, the visualization of circular split networks is not straightforward \cite{gh}. The \emph{circular network} algorithm creates a planar split network, beginning with a star graph, and at each iteration including an additional split.  Unfortunately, the structure of resultant split network depends on the order in which these splits are added, as shown by Figure~\ref{f:order}.
The top row of the figure shows the star graph with three splits $\{123, 456\}$, $\{126, 345\}$, $\{156, 234\}$ imposed on it; the bottom row shows the same three splits applied in an alternate ordering, resulting in a different diagram.

\begin{figure}[h]
\includegraphics{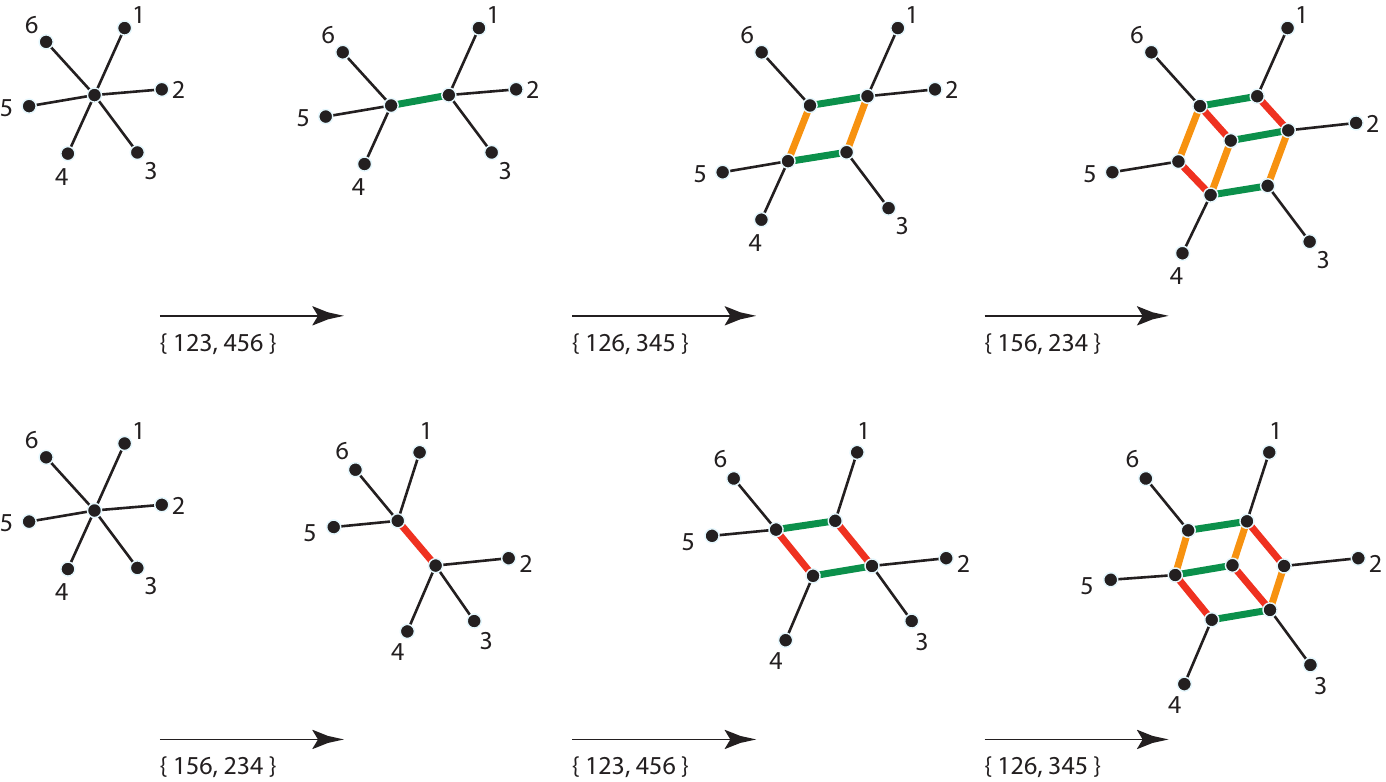}
\caption{The circular split network is not commutative with respect to the order of the splits.}
\label{f:order}
\end{figure}

In order to avoid issues with commutativity of split order, we use the  \emph{dual polygonal representation} for the duration of the paper. Given a circular split system with some circular ordering $\pi$ of the species, consider a regular $n$-gon, with the edges cyclically labeled with $\pi$.  For each split, draw a diagonal partitioning the appropriate edges. 
Figure \ref{f:polyversion} shows the relationship between the traditionally drawn split network and its corresponding dual polygonal representations.
For weighted networks, lengths of (parallel) edges are now represented by weights on the diagonals.
Moreover, diagonals representing tree-like edges are \emph{noncrossing} and network-like edges are \emph{crossing}.

\begin{figure}[h]
\includegraphics{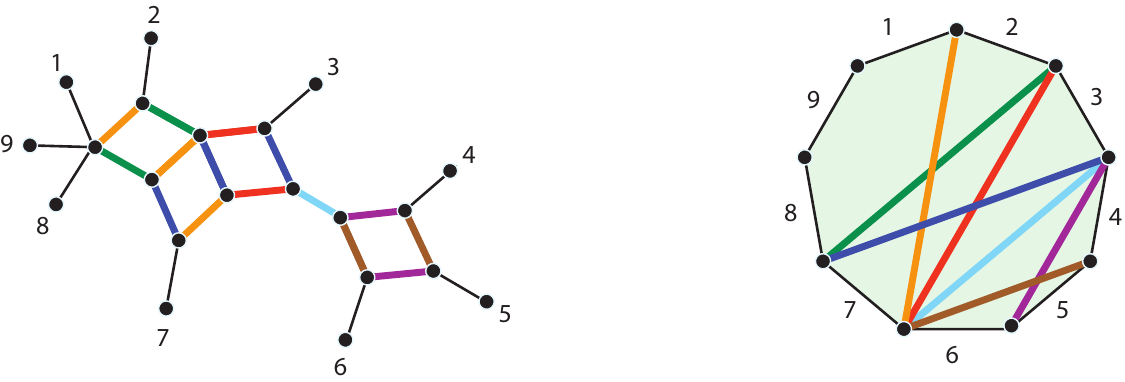}
\caption{The dual polygon representations of a circular split network.}
\label{f:polyversion}
\end{figure}

%%%%%%%%%%%%%%%%%%%%%%%%%%%%%%%%%%%%%%%%%%%%%%%%%%%%%%%%%%%%%%%%%%%%%%%%%%%%%%%%%%%%%%%%%
\subsection{}

A labeled polygon with diagonals gives rise to a circular split system, where each diagonal is a split of the system.  Conversely, a circular split system may be compatible with several different circular orderings, and there is a unique polygonal representation for each such ordering.  A simple geometric operation allows the discovery of all such compatible orderings. 

\begin{defn}
A \emph{twist} along a noncrossing diagonal $d$ of a labeled polygon $P$ is obtained by `breaking' $P$ along $d$ into two subpolygons, 
`reflecting' one of these pieces, and `gluing' them back (Figure~\ref{f:dtwist}).  
\end{defn}

\begin{figure} [h]
\includegraphics[width=\textwidth]{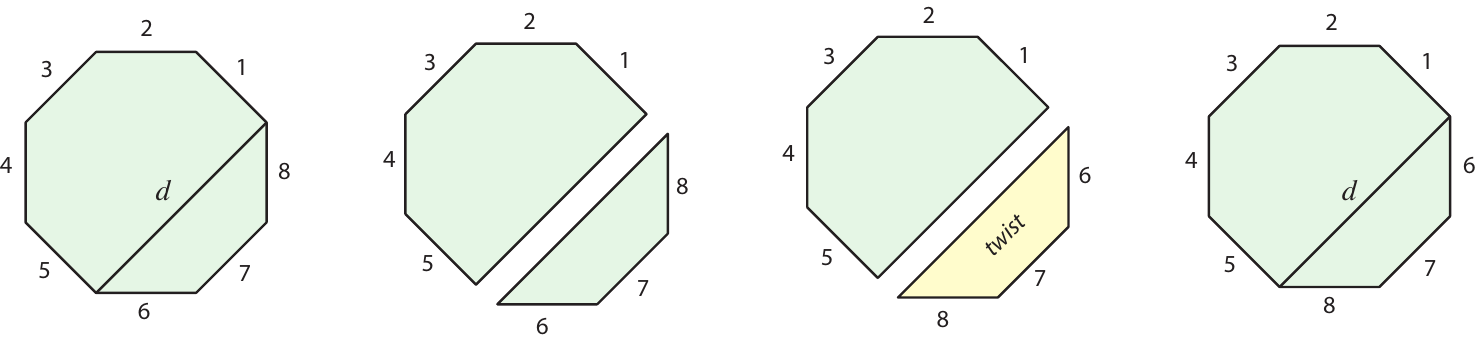}
\caption{Twisting along diagonal $d$.}
\label{f:dtwist}
\end{figure}

\begin{thm} \label{t:twist} 
Two distinct circular orderings $\pi_1$ and $\pi_2$ are compatible with a circular split system $\Sp$ if and only if beginning with the polygonal representation of $\Sp$ with $\pi_1$ there exist a sequence of twists along noncrossing diagonals that results in the polygonal representation of $\Sp$ with $\pi_2$.
\end{thm}

\begin{rem}
We allow two types of noncrossing diagonals: (i) diagonal representing a split in $\Sp$ that is noncrossing with other diagonals in the polygonal representation, or (ii) diagonal not representing a split in $\Sp$, but if drawn, would be noncrossing with other diagonals in the polygonal representation.
\end{rem}

\begin{proof}
The backwards direction is immediate: twisting along a noncrossing diagonal does not change any split partitions.   For the forward direction, begin with a representation of $\Sp$ with $\pi_1$.  Let $f(x)$ be the label of the edge adjacent to $x$ in the clockwise direction in $\pi_2$. %We begin the construction as follows, for $x=1$ and $\pi= \pi_1$.
\begin{enumerate}
\item If $f(x)$ is adjacent to $x$ in $\pi_1$ in the clockwise direction, do nothing.
\item If $f(x)$ is adjacent to $x$ in $\pi_1$ in the counterclockwise direction, and no diagonals exist between the two labels, draw a diagonal containing $x$ and $f(x)$, and twist the side containing $x$ and $f(x)$. 
\item If $f(x)$ is adjacent to $x$ in $\pi_1$ in the counterclockwise direction, and $d$ is a diagonal landing between $x$ and $f(x)$, twist along both sides of $d$. 
\item Otherwise, draw diagonal $d$ connecting the right endpoint of $x$ with the right endpoint of $f(x)$, and twist the side containing $f(x)$. The compatibility of $\pi_2$ with $\Sp$ guarantees $d$ to be noncrossing. 
\end{enumerate}
Repeating this process with the resulting circular ordering, now letting $x$ to be $f(x)$, produces the desired result after $n-2$ iterations.
\end{proof}

\begin{exmp}
The polygonal representation in Figure~\ref{f:twistequiv}(a) is twisted along a noncrossing diagonal of type (ii) in (b) resulting in (c), an equivalent polygonal representation of the same split system.  Similarly, a noncrossing diagonal of type (i) in (c) is twisted in (d) resulting in  (e), another equivalent representation.
\end{exmp}

\begin{figure} [h]
\includegraphics[width=\textwidth]{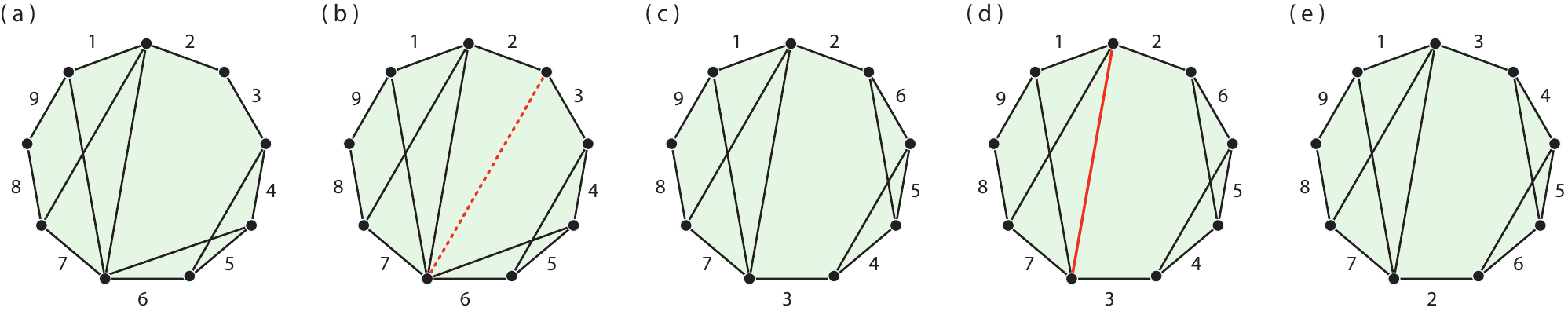}
\caption{Equivalent polygonal representations due to twisting.}
\label{f:twistequiv}
\end{figure}

%%%%%%%%%%%%%%%%%%%%%%%%%%%%%%%%%%%%%%%%%%%%%%%%%%%%%%%%%%%%%%%%%%%%%%%%%%%%%%%%%%%%%%%%%
%                                                      
%                Trees and Networks
%
%%%%%%%%%%%%%%%%%%%%%%%%%%%%%%%%%%%%%%%%%%%%%%%%%%%%%%%%%%%%%%%%%%%%%%%%%%%%%%%%%%%%%%%%%
\section{Trees and Networks}
\label{s:space}
\subsection{}

Billera, Holmes, and Vogtmann \cite{bhv} constructed an elegant space $\BHV{n}$ of isometry classes of metric trees with $n$ labeled leaves.\footnote{Classically, this space is defined in terms of trees with $n$ leaves and one root, whereas we consider unrooted trees with $n$ leaves.  Thus there is an index shift of one.}  Each such tree specifies a point in the orthant $[0, \infty)^{n-3}$, parametrized by the lengths of its internal edges, and thus defines coordinate charts for the space of such trees. The space $\BHV{n}$ is assembled by gluing $(2n-5)!!$ such orthants, one for each different binary tree on $n$ leaves \cite{dh}.  Two orthants of $\BHV{n}$ share a wall if and only if their corresponding binary trees differ by a \emph{rotation}, a move which collapses an interior edge of a binary tree, and then expands the resulting degree-four vertex into a different binary tree.  Figure~\ref{f:bhv4}(a) shows $\BHV{4}$ consisting of three (orthant) rays glued at the origin, the degenerate tree with no internal edges.

\begin{figure}[h]
\includegraphics[width=.9\textwidth]{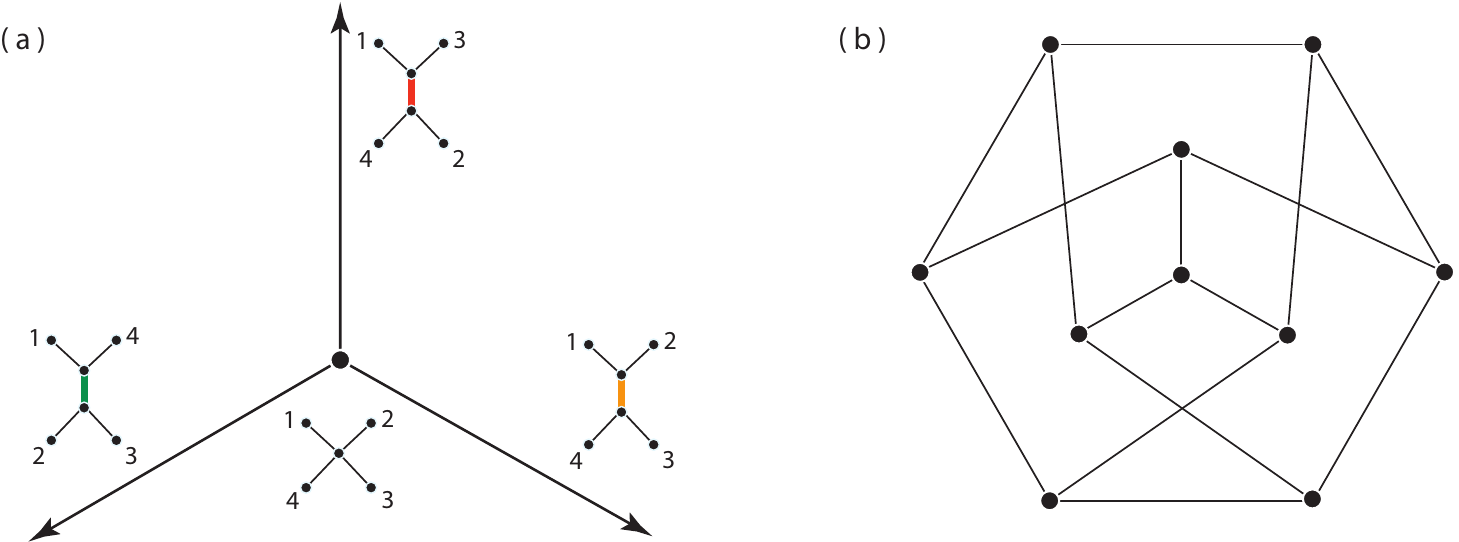}
\caption{(a) Tree space $\BHV{4}$ and (b) the simplicial complex $\T{5}$.}
\label{f:bhv4}
\end{figure}

\begin{defn}
Let $\T{n}$ be the subspace of $\BHV{n}$ consisting of trees with internal edge lengths that sum to 1.   It  is a pure simplicial $(n-3)$-complex composed of $(2n-3)!!$ chambers, with two adjacent chambers differing by a rotation of their underlying trees.  
%In particular, $\T{n}$ has one $(k-1)$-simplex for every tree with $k$ interior edges.
\end{defn}

Indeed, $\BHV{n}$ is a cone over this space (and thus contractible), where the cone-point is the degenerate tree with no internal edges.  For example, $\BHV{5}$ consists of 15 quadrants $[0, \infty)^2$ glued together, and its subspace $\T{5}$ is the Peterson graph with 15 edges, as displayed in Figure~\ref{f:bhv4}(b).  Here, the 10 vertices correspond to binary trees with five leaves and two internal edges.  

%%%%%%%%%%%%%%%%%%%%%%%%%%%%%%%%%%%%%%%%%%%%%%%%%%%%%%%%%%%%%%%%%%%%%%%%%%%%%%%%%%%%%%%%%
\subsection{}

We now introduce a space $\sn{n}$ of isometry classes of metric circular split networks with $n$ labeled leaves.  Each network specifies a point in this network space, parametrized by the weights of its splits, or equivalently the lengths of its (parallel sets of) internal edges. The space is assembled by gluing $(n-1)!/2$ orthants together, each of which corresponds to a unique circular ordering of the $n$ species, up to rotation and reflection. Therefore, each orthant has dimension $n(n-3)/2$, the maximal number of splits compatible with a given circular ordering, or the maximal number of diagonals on an $n$-gon.  Orthants glue together along cells that represent split systems that are compatible to the orderings of their respective chambers, given by Theorem~\ref{t:twist}.

There is a natural coordinate system based on splits into which $\sn{n}$ embeds. Let 
\begin{equation}
\label{e:split}
\delta \ := \ 2^{n-1} - n - 1 \, ,
\end{equation}
enumerating networks with exactly one nontrivial split, obtained by considering half the number of partitions $2^n$ (compensating for double-counting) and subtracting the set of $n$ trivial splits and $1$, the empty set.  Each circular split network is an element of $\R^\delta$, defined by the set of its splits and parametrized by the lengths of its internal edges.  However, a circular split network will have at most $n(n-3)/2$ nonzero values in this coordinate system, based on the maximal number of diagonals on an $n$-gon.  In summary, we have:

\begin{prop} \label{p:embed}
The simplicial fan $\sn{n}$ naturally embeds into $\R^\delta$.
\end{prop}

\begin{rem}
Recently, Hacking, Keel, and Tevelev \cite{hkt} have shown how to build this simplicial fan $\sn{n}$ canonically from the root system $D_n$ based on real algebraic geometry of del Pezzo surfaces.  Indeed, precise analogs for $E_6, E_7, E_8$ are also provided, based on work by Sekiguchi and Yoshida \cite{sy}.
\end{rem}

\begin{figure}[h]
\includegraphics[width=.9\textwidth]{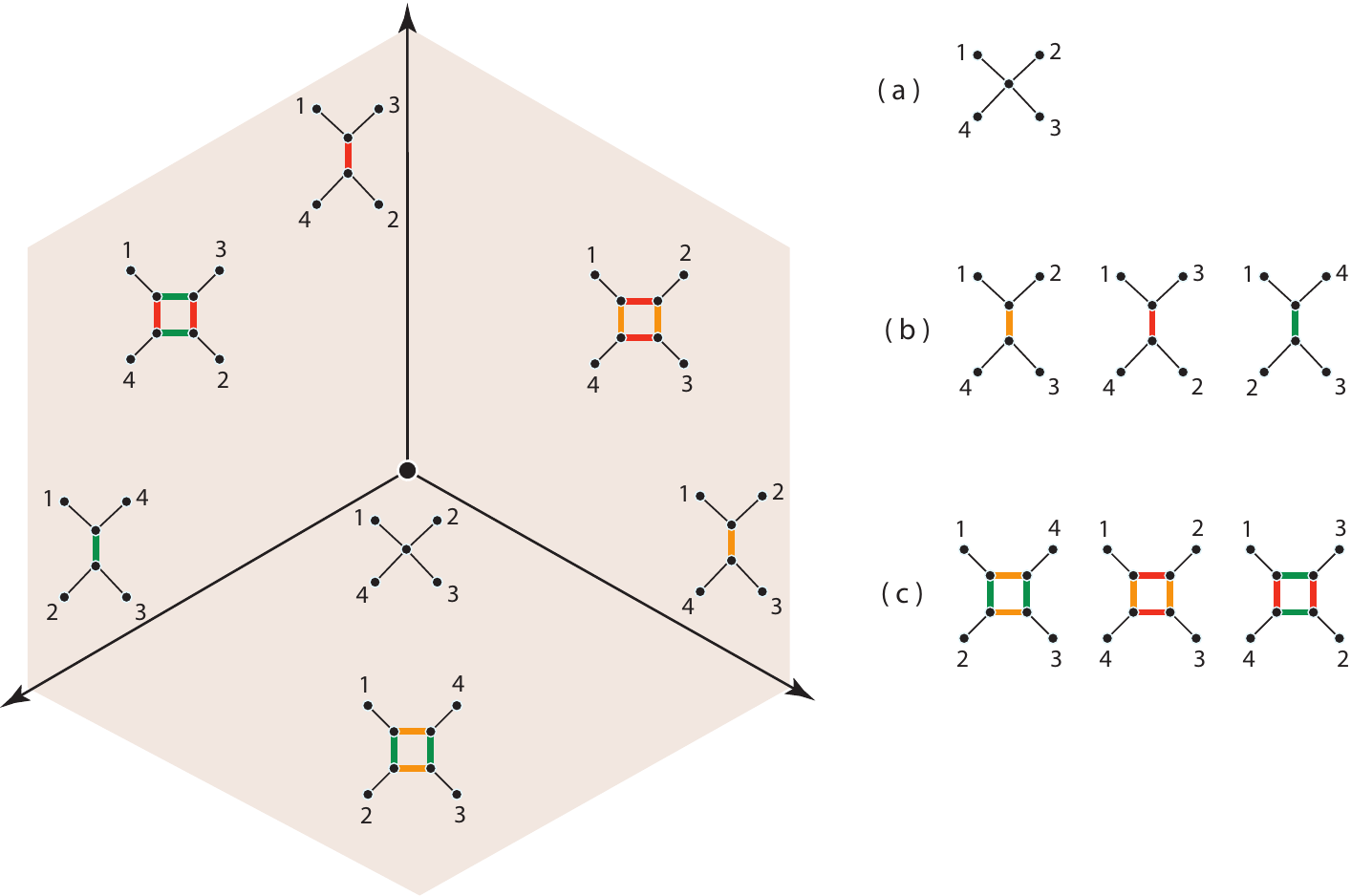}
\caption{The space of circular split networks for $n=4$.}
\label{f:csn4}
\end{figure}

\begin{exmp}
Figure~\ref{f:csn4} shows the case for $\sn{4}$, the space of metric circular split networks for four species. Three quadrants $\R^2_{\geq 0}$ tile this space (c), corresponding to the distinct circular orderings of the four labels.  These quadrants glue along the three boundary rays (b), where all meet at the origin, the degenerate network with no internal edges (a).  Thus $\sn{4}$ is homeomorphic to the plane $\R^2$.  Note that $\BHV{4}$ from Figure~\ref{f:bhv4}(a) is a natural subspace of $\sn{4}$.
\end{exmp}

%%%%%%%%%%%%%%%%%%%%%%%%%%%%%%%%%%%%%%%%%%%%%%%%%%%%%%%%%%%%%%%%%%%%%%%%%%%%%%%%%%%%%%%%%
\subsection{}
Indeed, $\BHV{n}$ is simply the subspace of $\sn{n}$ restricted to \emph{pairwise compatible} split systems.   Another natural subspace of $\sn{n}$ is the \emph{link of the origin} $\n{n}$, the union of the set of points in each orthant with internal edge lengths of networks that sum to 1.  Since the set of such points in a single orthant forms a simplex, the following is immediate:
 
\begin{prop}
Network space $\n{n}$ is a connected simplicial complex of dimension \break \mbox{$(n(n-3)/2) - 1$}, with one $k$-simplex for every labeled $n$-gon with $k+1$ diagonals.
\end{prop}

\noindent Figure~\ref{f:n4} shows $\n{4} \subset \sn{4}$, the 1-dimensional simplicial complex formed by gluing three edges together, forming a triangle, parametrizing circular networks whose internal edge lengths sum to 1.  

\begin{figure}[h]
\includegraphics[width=\textwidth]{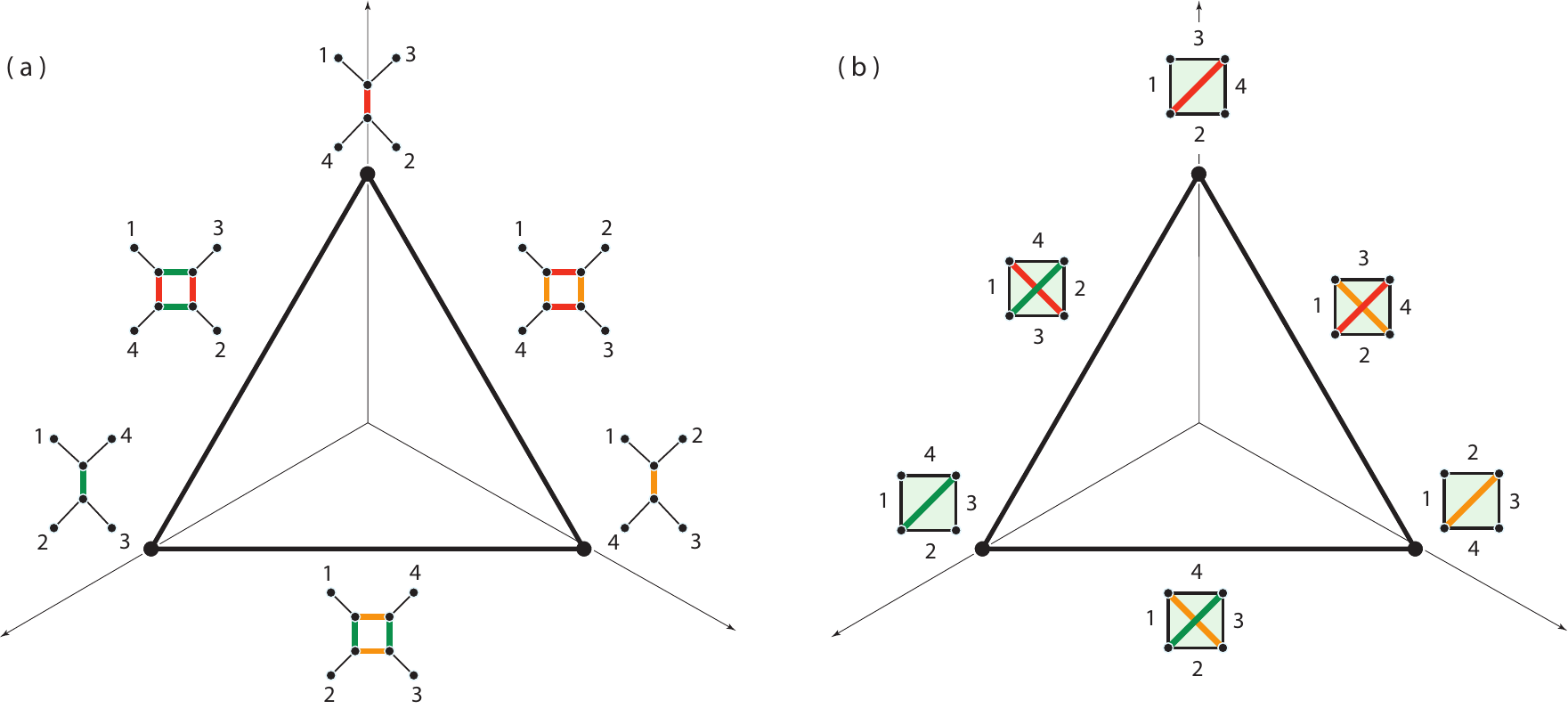}
\caption{The simplicial complex (a) $\n{4}$ and (b) its polygonal representation.}
\label{f:n4}
\end{figure}

It was shown in \cite{bhv} that $\BHV{n}$ was a CAT(0) space, ensuring that any two points had a unique geodesic between them.  This is not the case for networks, however, due to the underlying geometry of $\n{n}$.

\begin{prop}
The space $\sn{n}$ of circular split networks is not \textup{CAT(0)}.
\end{prop}

\begin{proof}
Subdivide each orthant of $\sn{n}$ into unit cubes having integral vertices, making $\sn{n}$ into a cubical complex. 
By a theorem of Gromov \cite{gro}, a cubical complex is CAT(0) if and only if it is simply connected and the link at every vertex is a flag complex.  
%Since $\sn{n}$ is a cone, it is contractible, and therefore simply connected.  
But Figure~\ref{f:empty} shows the link at the origin $\n{n}$ having \emph{empty triangles}, and therefore failing to be flag.  Here, the three edges of the triangle exist but there is no 2-simplex (represented by a polygon with three diagonals) bounding this triangle.
\end{proof}

\begin{figure}[h]
\includegraphics[width=.4\textwidth]{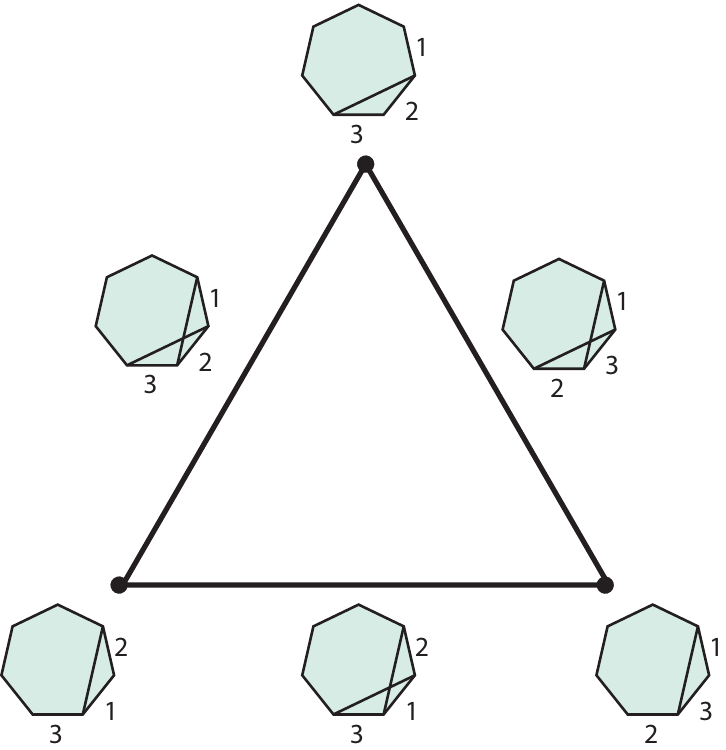}
\caption{Empty triangles in $\sn{n}$.}
\label{f:empty}
\end{figure}

%%%%%%%%%%%%%%%%%%%%%%%%%%%%%%%%%%%%%%%%%%%%%%%%%%%%%%%%%%%%%%%%%%%%%%%%%%%%%%%%%%%%%%%%%
%                                                      
%                Properties of  Network Space
%
%%%%%%%%%%%%%%%%%%%%%%%%%%%%%%%%%%%%%%%%%%%%%%%%%%%%%%%%%%%%%%%%%%%%%%%%%%%%%%%%%%%%%%%%%
\section{Combinatorics of Network Space}
\label{s:props}
\subsection{}

We explore some properties of the simplicial complex $\n{n}$.  For $n=4$, $\n{n}$ is a triangle, with three vertices and three edges.  Enumeration for larger values of $n$ follows:

\begin{thm}
For $n > 4$, the network space $\n{n}$  has $c = (n-1)!/2$ simplicial chambers of dimension $d = n(n-3)/2 - 1$.  It has $c(d+1)$ ridges, $\delta$ vertices and $\binom{\delta}{2}$ edges.  
\end{thm}

\begin{proof}
The chambers are the different ways of labeling an $n$-gon with the maximal set of diagonals. Since there are $n(n-3)/2$ diagonals on an $n$-gon, the number of ridges (codim 1 faces) corresponds to removing one of these diagonals from the maximal set.  The vertices enumerate the set of networks with one split, obtained in Eq.~\eqref{e:split} above. And since each pair of splits is contained in a circular split system, the enumeration of edges follows.
\end{proof}

\begin{cor} \label{c:complete}
The 1-skeleton of $\n{n}$ is the complete graph on $\delta$ vertices.
\end{cor}

Note that the number of chambers in which a $k$-cell of $\n{n}$ resides depends on the structure of polygon corresponding to the cell, not the dimension of the cell itself. A $k$-cell may be part of only one chamber or it might be part of several chambers of $\n{n}$. Figure~\ref{f:types} (a) - (c) depict polygon corresponding to a $14$-cell, $9$-cell, and $10$-cell, respectively, that are contained in exactly two distinct chambers of $\n{8}$.  Part (d) depicts another $10$-cell belonging to a unique chamber of $\n{8}$.

\begin{figure}[h]
\includegraphics[width=.9\textwidth]{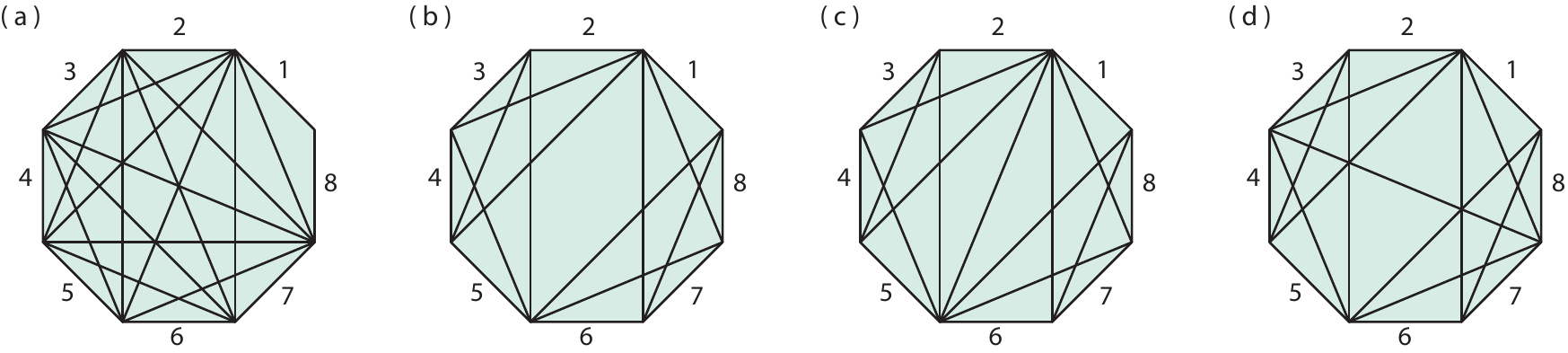}
\caption{Cells (a), (b), (c) belong to two distinct chambers, whereas cell (d) to a unique chamber of $\n{8}$.}
\label{f:types}
\end{figure}

\begin{thm} 
\label{t:glue}
Two chambers of $\n{n}$ can intersect along a face of at most dimension 
$$(n-3)(n-2)/2  - 1.$$
\end{thm}

\begin{proof} 
By Theorem~\ref{t:twist}, it is necessary that there exists a possible noncrossing diagonal of the polygonal representation for two distinct circular ordering to be compatible.  An $n$-gon can have at most $n(n-3)/2$ diagonals, and $n-3$ of them must be removed in order to make room for a noncrossing diagonal $d$.  Indeed, there must be $x$ vertices of one side of $d$ and $y$ vertices on the other such that $x+y=n-2$.  It is necessary to remove $xy$ diagonals, which is minimized when $x=n-3$ and $y=1$, and the result follows.
\end{proof}

%%%%%%%%%%%%%%%%%%%%%%%%%%%%%%%%%%%%%%%%%%%%%%%%%%%%%%%%%%%%%%%%%%%%%%%%%%%%%%%%%%%%%%%%%
\subsection{}

The space $\n{5}$ is a 4-dimensional simplicial complex, composed of 12 chambers, each corresponding the unique circular orderings on five species labels.  There are seven different types of cells, corresponding to split networks with different structures, as outlined in Figure~\ref{f:celltypes}:  The rows denote dimension, polygon representation, network type, and enumeration in $\n{5}$, respectively. There are two distinct 2-cells (orange, yellow), each bounded by two types of edges (red, blue), detailed in the left side of Figure~\ref{f:n5}.

\begin{figure}[h]
\includegraphics[width=.9\textwidth]{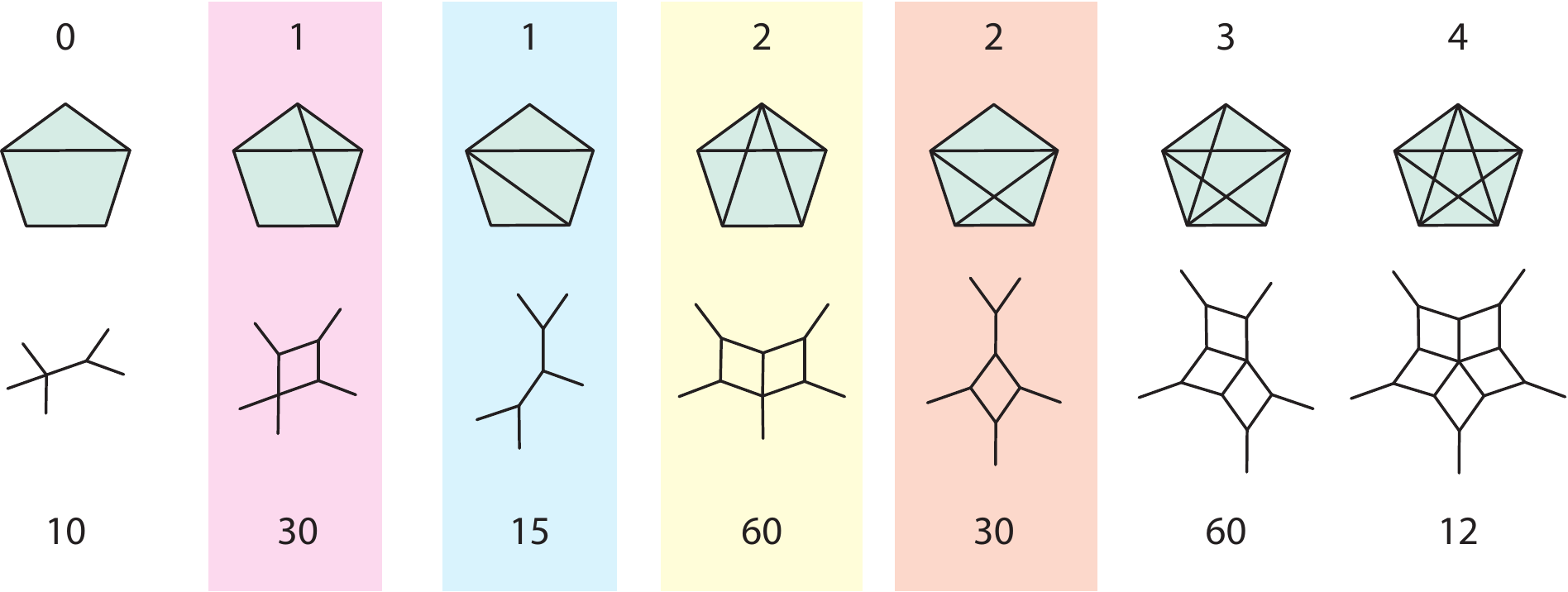}
\caption{The stratification of $\n{5}$ by distinct cell types, where the number of such cells is enumerated in the last line.}
\label{f:celltypes}
\end{figure}

The right side of Figure~\ref{f:n5} illustrates a (4-simplex) chamber in $\n{5}$.  Of the ten triangles appearing as faces, five are orange and five are yellow; two of each color are highlighted in the illustration.  Interestingly, the triangles of each color form a mob\"ius strip within the 4-simplex.  By Theorem~\ref{t:glue}, two chambers of $\n{5}$ meet only along vertices, edges, and triangles, and not along any of the tetrahedra.

\begin{figure}[h]
\includegraphics[width=\textwidth]{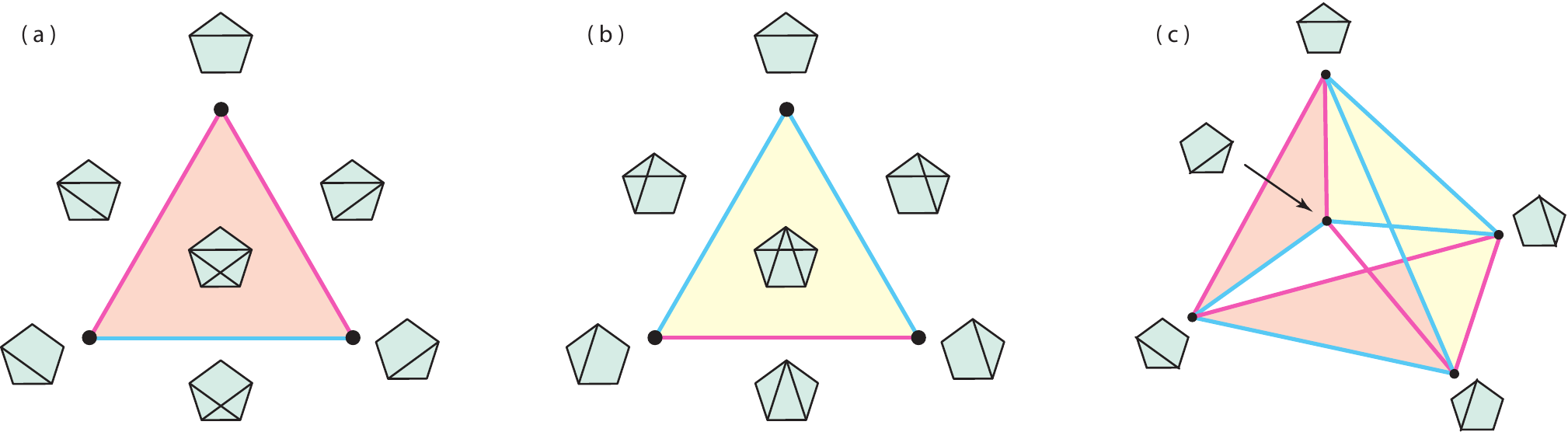}
\caption{The two types of triangles along with a partial labeling of a chamber in $\n{5}$.}
\label{f:n5}
\end{figure}

The global structure of the 1-skeleton of $\n{5}$ is shown in Figure~\ref{f:n5-skeleton}(a), with the coloring based on Figure~\ref{f:celltypes}. The red edges in (b) form the Peterson Graph, tree space $\T{5}$ from Figure~\ref{f:bhv4}(b), and the blue edges in (c) form the skeleton of a rectified 5-cell. 

\begin{figure}[h]
\includegraphics[width=\textwidth]{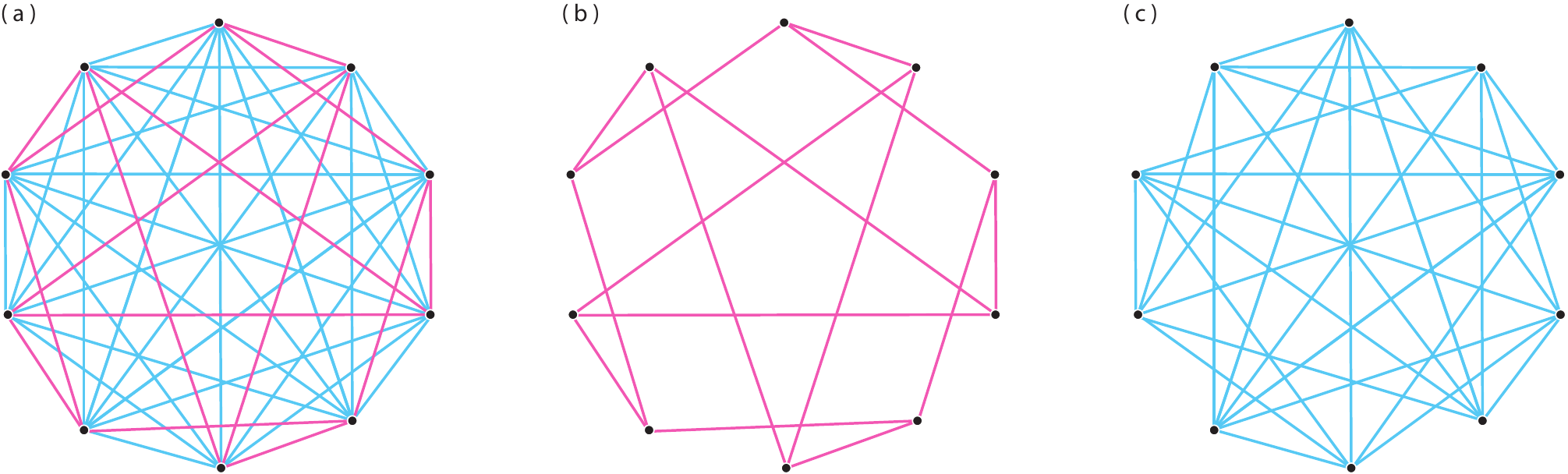}
\caption{The 1-skeleton of $\n{5}$, color-coded based on Figure~\ref{f:celltypes}.}
\label{f:n5-skeleton}
\end{figure}

%%%%%%%%%%%%%%%%%%%%%%%%%%%%%%%%%%%%%%%%%%%%%%%%%%%%%%%%%%%%%%%%%%%%%%%%%%%%%%%%%%%%%%%%%
%                                                      
%                Moduli Spaces
%
%%%%%%%%%%%%%%%%%%%%%%%%%%%%%%%%%%%%%%%%%%%%%%%%%%%%%%%%%%%%%%%%%%%%%%%%%%%%%%%%%%%%%%%%%
\section{Moduli spaces}
\label{s:moduli}
\subsection{}

An elegant polytope captures the structure of the space of planar \emph{rooted} trees:

\begin{defn} 
The \emph{associahedron} is a convex  polytope of dimension $n-2$ whose face poset is isomorphic to that of bracketings of $n$ letters, ordered so $a \prec a'$ if $a$ is obtained from $a'$ by adding new brackets.
\end{defn}

The  associahedron was constructed independently by Haiman (unpublished) and Lee \cite{lee}, though
Stasheff had defined the underlying abstract object twenty years previously, in his work on associativity in
homotopy theory \cite{sta}.  Figure~\ref{f:assoc}(a) shows the 2D associahedron $K_4$ with a labeling of
its faces, and (b) shows the 3D version $K_5$.
\begin{figure}[h]
\includegraphics{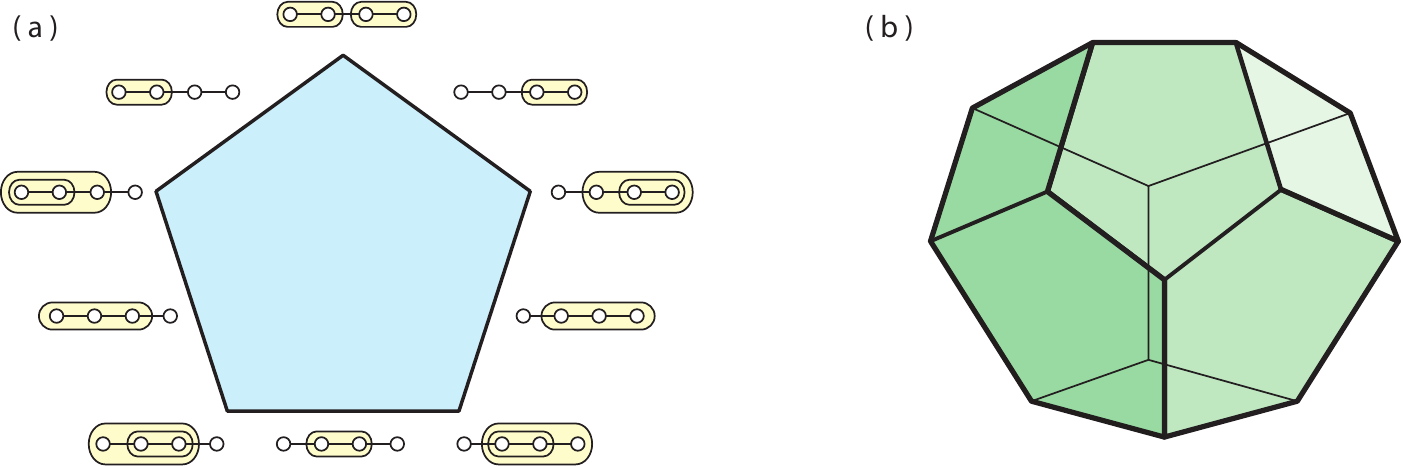}
\caption{Associahedra $K_4$ and $K_5$.}
\label{f:assoc}
\end{figure}
There are over a hundred combinatorial and geometric interpretations \cite{stan} of the Catalan number
\begin{equation} \label{e:catalan}                            
C_{n-1} \ = \ \frac{1}{n}\binom{2n-2}{n-1} \ ,
\end{equation} 
which index the vertices of the associahedron $K_n$. Most important to us is the relationship between bracketings of $n$ letters, rooted trees with $n$ species, and polygons with noncrossing diagonals, shown in Figure~\ref{f:assoc-tree}.
In particular, a codim $k$ face of the associahedron $K_{n-1}$ is associated to an (unlabeled) $n$-gon with $k$ noncrossing diagonals. 

\begin{figure}[h]
\includegraphics{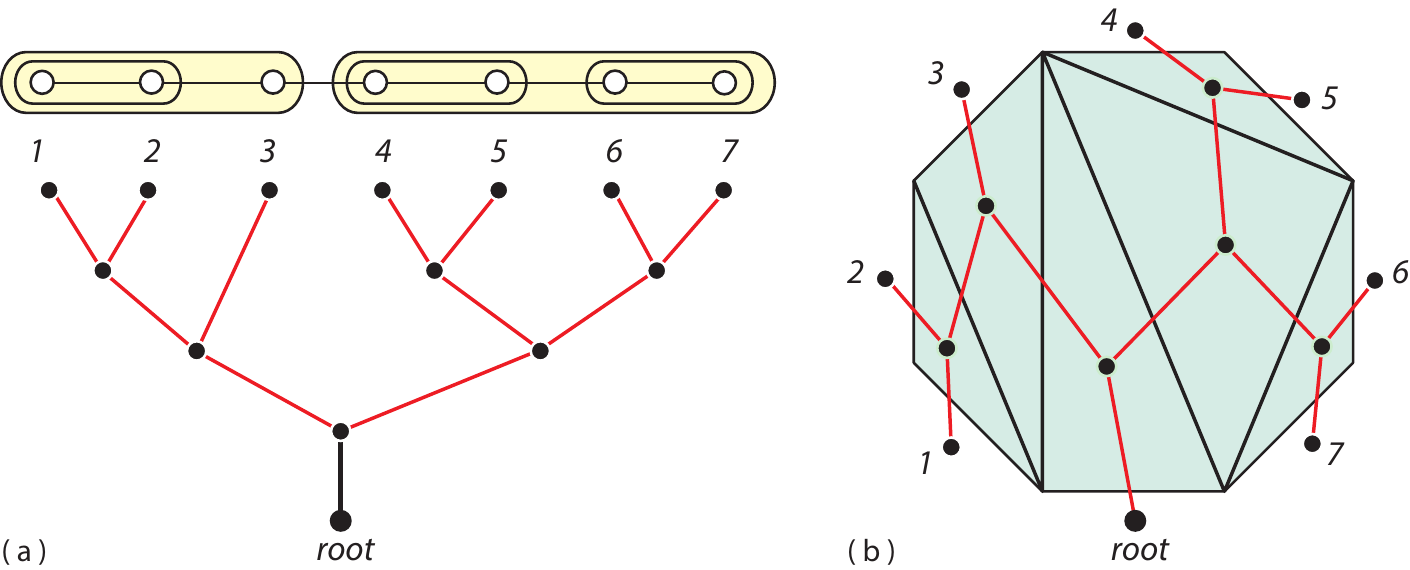}
\caption{(a) Bracketings and planar trees and (b) polygons with diagonals.}
\label{f:assoc-tree}
\end{figure}
 
%%%%%%%%%%%%%%%%%%%%%%%%%%%%%%%%%%%%%%%%%%%%%%%%%%%%%%%%%%%%%%%%%%%%%%%%%%%%%%%%%%%%%%%%%
\subsection{}

The moduli problem for algebraic curves has been a central problem in mathematics since Riemann.  In the 1970s 
it was solved over the integers $\Z$ by Deligne, Mumford, Knudsen \cite{git} and others, where a 
special case constructs a moduli space for \emph{real} algebraic curves of genus zero
marked with distinct smooth points.  That solution can be regarded as a good compactification \M{n}
of the space
$$\oM{n} \ = \ \Conf^{n}(\RP)/\PGL$$
of $n+1$ distinct particles on the real projective line. 

\begin{thm} \cite[Section 4]{dev1}
\label{t:tiling}
The moduli space \M{n} is tiled by $(n-1)!/2$ copies of $K_{n-1}$, one for each labeling of an $n$-gon, up to rotation and reflection.  Two faces of associahedra, represented by labeled polygons with noncrossing diagonals $P_1$ and $P_2$, are identified in \M{n} if twisting along certain diagonals of $P_1$ yields $P_2$.
\end{thm}

\noindent The moduli space \M{3} is a point, and the manifold \M{4} is homeomorphic to a circle, with the {\em cross-ratio} serving as the homeomorphism; it is tiled by three $K_3$ line segments, glued together to form a triangle, shown in Figure~\ref{f:m04}.  

\begin{figure}[h]
\includegraphics{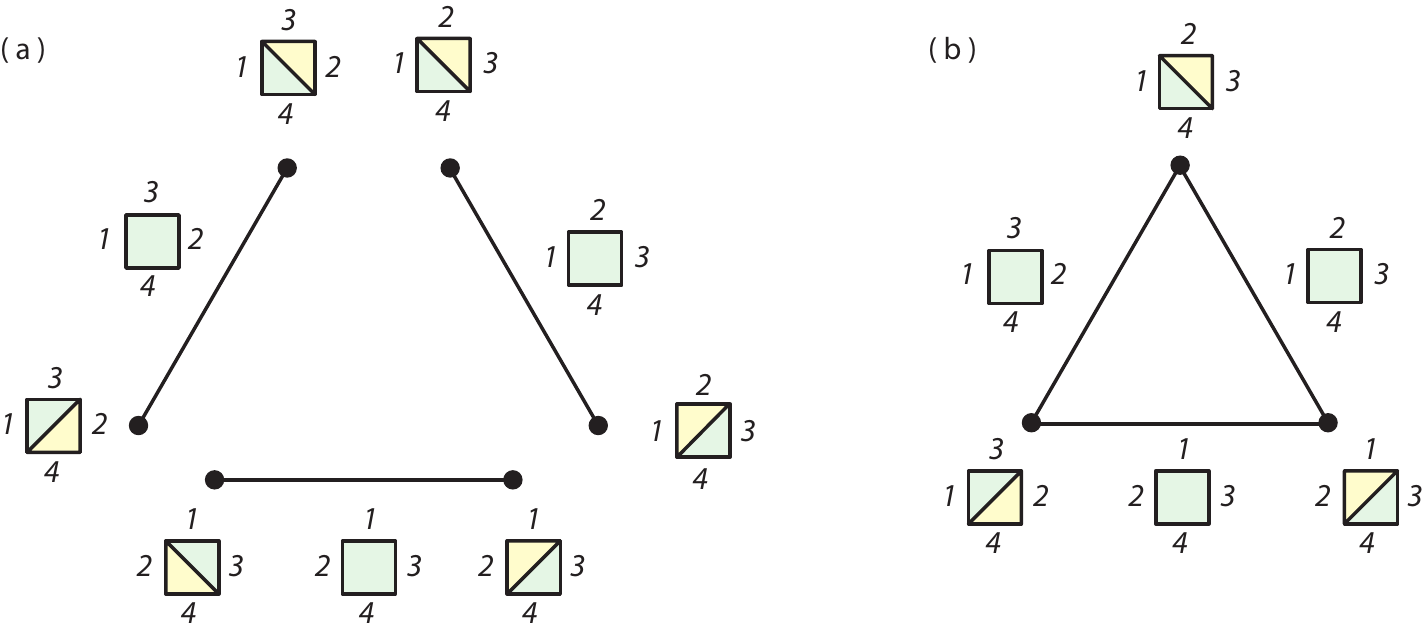}
\caption{Gluing three $K_3$ using twists to form \M{4}.}
\label{f:m04}
\end{figure}

\begin{exmp} 
An illustration of \M{5} appears in Figure~\ref{f:m0506}(a), resulting in the connected sum of a torus with three real projective planes, tiled by 12 associahedra $K_4$ from Figure~\ref{f:assoc}(a).  Part (b) shows the example for \M{6}, tiled by 60 copies of $K_5$.
\end{exmp}

\begin{figure}[h]
\includegraphics{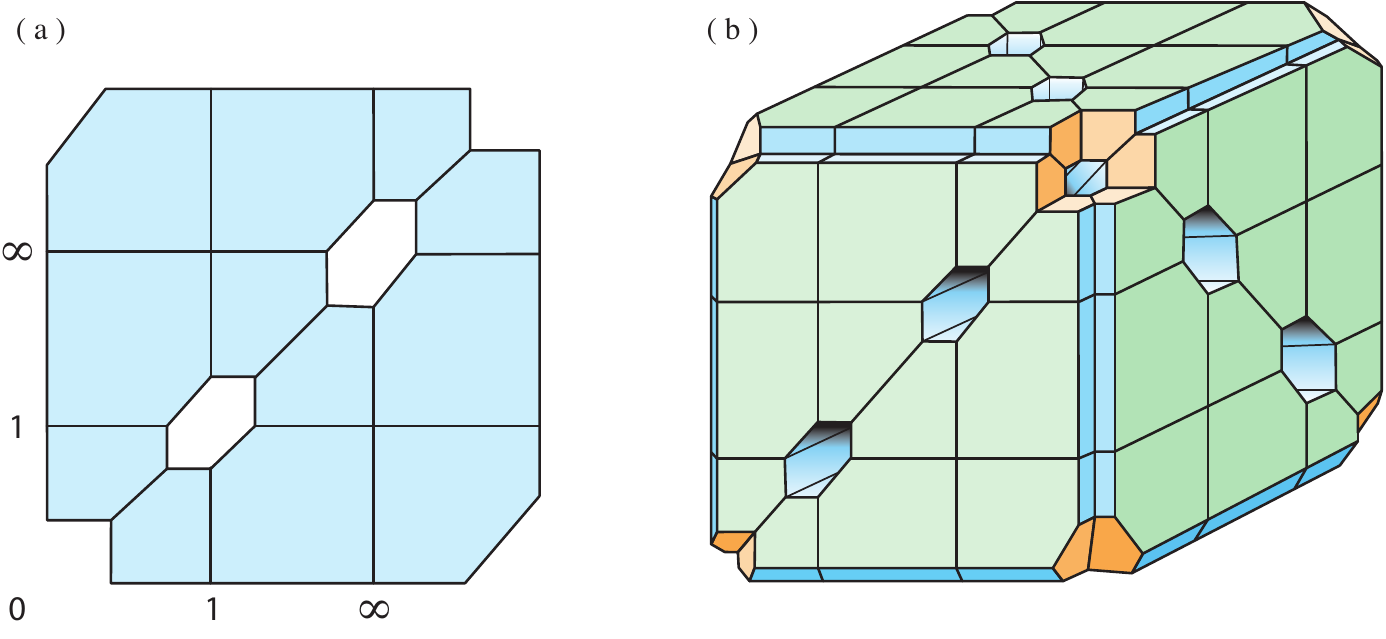}
\caption{(a) \M{5} and (b) \M{6} as blowups of tori.}
\label{f:m0506}
\end{figure}

%%%%%%%%%%%%%%%%%%%%%%%%%%%%%%%%%%%%%%%%%%%%%%%%%%%%%%%%%%%%%%%%%%%%%%%%%%%%%%%%%%%%%%%%%
\subsection{}

A natural embedding $\Phi$ of the moduli space \M{n} into the network space $\n{n}$ is now provided. Throughout this section, we fix some labeling of the $n$-gon thereby choosing a chamber of \M{n} and $\n{n}$.  The canonical coordinates for $\n{n}$ is given in Proposition~\ref{p:embed}, one for each possible split of $n$ species.  For a chosen labeling however, recall that at most $n(n-3)/2$ of these coordinates will have nonzero values, one for each diagonal of the $n$-gon.\footnote{To be consistent with the language of associahedra, we refer to sets of splits as sets of diagonals throughout this section.} Thus, the coordinate system in a chamber of $\n{n}$ is
\begin{equation}
\label{e:coords}
(x_1, \ x_2, \ \dots, \ x_{n(n-3)/2}) \, ,
\end{equation}
a dimension $x_i$ for each diagonal $d_i$ of the $n$-gon.

Consider the set of vertices $V(K_{n-1})$ of the associahedron, each a triangulation with exactly $n-3$ diagonals.  For each vertex $v$, assign the $d$-th coordinate of the map $\Phi(v)$ as
\begin{equation}
\label{e:facec}
\Phi_d(v) \ := \ 
\begin{cases}
    1/(n-3) & \ \ \ \text{if $d$ is a diagonal of $v$} \\
    0 & \ \ \ \text{otherwise}.
\end{cases}
\end{equation}
For a face $f$ of the associahedron, let $V_f$ be the subset of vertices of $V(K_{n-1})$ incident to $f$.  Assign to the barycenter $v_f$ of each face $f$ the coordinate
\begin{equation}
\label{e:barycenter}
\Phi(v_f) \ := \ \text{centroid} \ \{\Phi(v) \suchthat v \in V_f\} \, .
\end{equation}

\begin{rem}
It follows that $\Phi(v)$ is an element of $\n{n}$, the sum of the coordinates being 1.
\end{rem}

\begin{exmp}
Figure~\ref{f:average} illustrates this computation for several barycenters of the associahedron $K_6$.  The top row shows the nine polygons with one diagonal, each an axes in the coordinate system.  The left column displays four different barycenters of the subdivision, along with their respective coordinates in this system.
\end{exmp}

\begin{figure}[h]
\includegraphics[width=.8\textwidth]{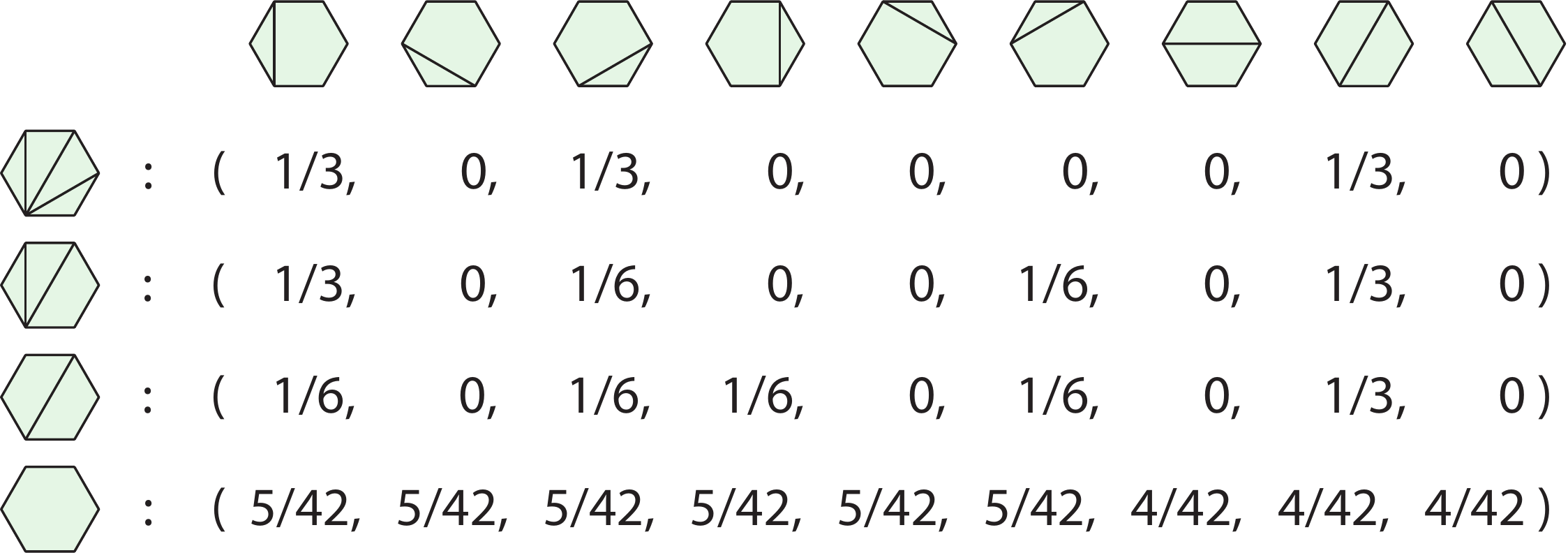}
\caption{Coordinates for barycenters.}
\label{f:average} 
\end{figure}

\begin{rem}
A \emph{flag} $F$ of the associahedron is a sequence $\{f_0, f_1, \dots f_{n-3}\}$ of faces such that $f_0 \subset f_1 \subset \dots \subset  f_{n-3}.$  We can reinterpret this as a sequence of subsets of diagonals, where the set $f_i$ contains exactly $n-3-i$ noncrossing diagonals.  
%Note that in the barycentric subdivision of $K_{n-1}$, each flag corresponds to an $(n-3)$-simplex.
\end{rem}

\begin{prop} \label{p:embed}
For flag $F$ of $K_{n-1}$, and $\Delta_F$ its associated simplex in its barycentric subdivision, the map
$$\Phi(\Delta_F) \ := \ \text{convex hull} \ \ \{\Phi(v_{f_0}), \ \Phi(v_{f_1}), \ \dots, \ \Phi(v_{f_{n-3}})\}$$
is an embedding of $\Delta_F$ into a chamber of $\n{n}$, where $v_{f_i}$ are the vertices of $\Delta_F$, the barycenters given in Eq.~\eqref{e:barycenter}.  This map $\Phi$ extends to the union of all the flags of $K_{n-1}$, providing an embedding of the associahedron into a chamber of $\n{n}$.
\end{prop}

The proof of this proposition is relegated to the next section.  We can extend this result for all of \M{n}.

\begin{thm} \label{t:moduliembed}
The map $\Phi$ naturally extends to an embedding of \M{n} into $\n{n}$.
\end{thm}

\begin{proof}
We have defined the coordinates for one particular chamber (labeling of the $n$-gon) of $\n{n}$, as defined in Eq.~\eqref{e:coords}.  This coordinate system naturally extends to the canonical coordinates for $\n{n}$ as given in Proposition~\ref{p:embed}, one for each possible split of $n$ species, with at most $n(n-3)/2$ of these coordinates  having nonzero values.  Note further that two associahedral chambers glue to form \M{n} by \emph{twisting}, as given in Theorem~\ref{t:tiling}, while their image in $\n{n}$ under $\Phi$ glue accordingly, due to Theorem~\ref{t:twist}.
\end{proof}

\begin{exmp}
Figure~\ref{f:n4-coord} illustrates the embedding of \M{4} in $\n{4}$, part (a) showing the polygonal labeling and (b) the coordinates of the vertices and the barycenters.  Each associahedral edge sits inside a chamber of $\n{4}$ with coordinates defined by Proposition~\ref{p:embed}.  The edges glue together as do the chambers of $\n{n}$, forming \M{4}. Compare this with Figure~\ref{f:n4}(b). 
\end{exmp}

\begin{figure}[h]
\includegraphics[width=\textwidth]{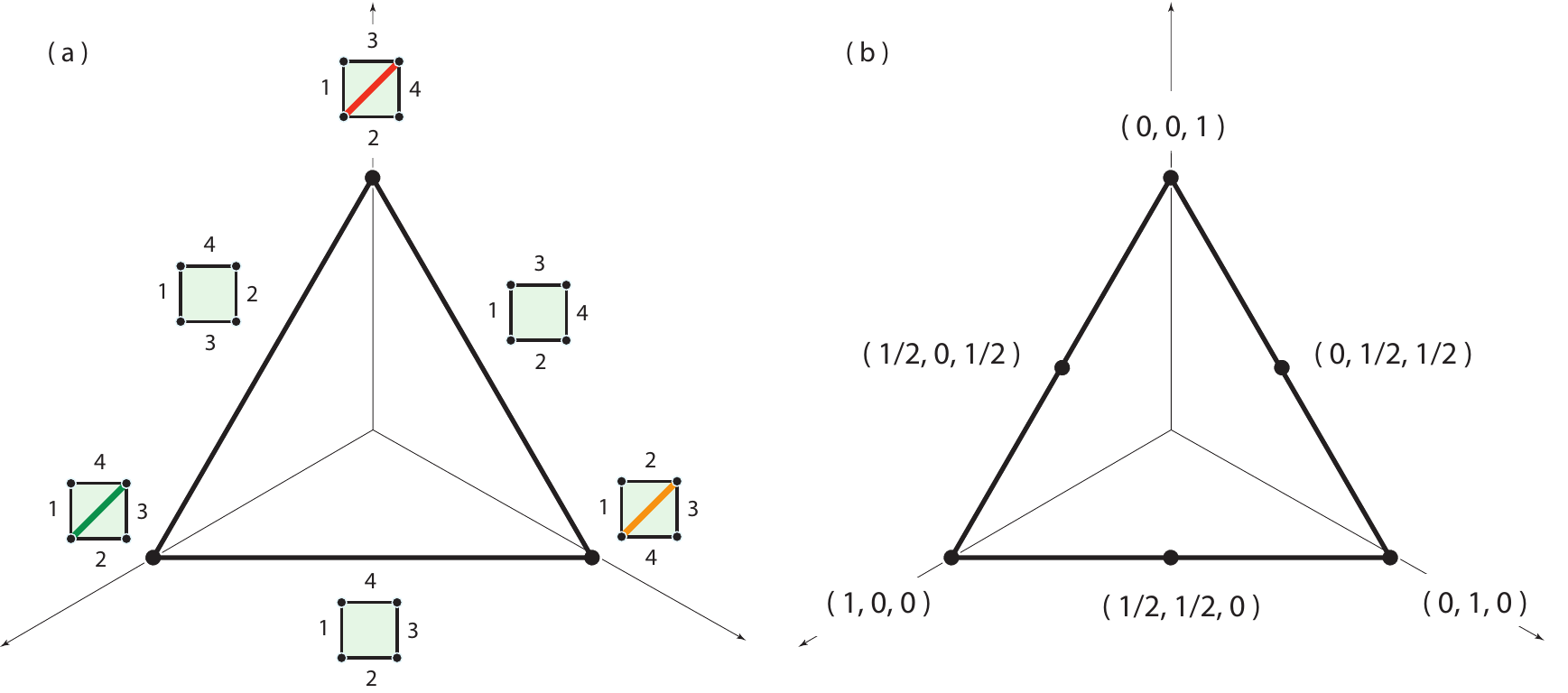}
\caption{Embedding of \M{4} in $\n{4}$ along with coordinates.}
\label{f:n4-coord} 
\end{figure}

\begin{rem}
It is interesting to wonder whether more can be said about the relationship between \M{n} and $\n{n}$.  For example, how do the homotopy types of these two spaces compare?  It was shown in \cite{djs} that \M{n} is \emph{aspherical}, its homotopy properties encapsulated in its fundamental group, whereas little is known about $\n{n}$.
\end{rem}

%%%%%%%%%%%%%%%%%%%%%%%%%%%%%%%%%%%%%%%%%%%%%%%%%%%%%%%%%%%%%%%%%%%%%%%%%%%%%%%%%%%%%%%%%
%                                                      
%                Embeddings
%
%%%%%%%%%%%%%%%%%%%%%%%%%%%%%%%%%%%%%%%%%%%%%%%%%%%%%%%%%%%%%%%%%%%%%%%%%%%%%%%%%%%%%%%%%
\section{Proof of Embeddings}  \label{s:embed}
\subsection{}

In this section, we prove Proposition~\ref{p:embed} which states that the map $\Phi$ defined above is an embedding. We show $\Phi$ is injective and preserves dimension in Lemma~\ref{l:newlem} and Lemma~\ref{l:simplex}, respectively, with embedding directly following. First, consider the following description of $\Phi$ based on the Catalan numbers $C_n$ from Eq.~\eqref{e:catalan}.

\begin{prop} \label{p:cp}
Let $f$ be a face of the associahedron with barycenter $v_f$, represented by a polygon $P$ with noncrossing diagonals.  
If $\Phi$ is the map defined in Eq.~\eqref{e:barycenter}, then the coordinate of \ $\Phi (v_f)$ associated to diagonal $d$ is given by 
\begin{equation}
\label{e:cat-count}
   \Phi_d (v_f) = \left\{
     \begin{array}{lr}
       \displaystyle{\frac{1}{n-3}} & \text{ if $P$ contains $d$} \vspace{.2in} \\ \vspace{.2in}
         0 & \text{ if $d$ crosses $P$}\\
         \displaystyle{\frac{1}{n-3}} \cdot \frac{C_{s-2} \ C_{t-2}}{C_{s+t-4}}&\text{otherwise}
     \end{array}
   \right.
\end{equation}
where the two polygons in $P$ on either side of $d$ are an $s$-gon and $t$-gon.
\end{prop}

\begin{proof} 
Recall that for a diagonal $d$ and a vertex $v$, the corresponding coordinate in the embedding $\Phi_d(v)$ is given in Eq.~\eqref{e:facec}. Since $\Phi_d(v_f)$ is the average (centroid) of the set $\{\Phi_d(v) \suchthat  v  \in V_f\}$, as shown in Eq.~\eqref{e:barycenter}, we have
\begin{align*} 
\Phi_d(v_f) \ & = \ \frac{1}{n-3} \ \left(\frac{ \text{\# vertices of $f$ that include diagonal $d$ }}{ \text{\# vertices of $f$}}\right) \\
& = \ \frac{1}{n-3} \ \left(\frac{ \text{\# triangulations that include $P$ and $d$ }}{ \text{\# triangulations that include $P$}}\right).
\end{align*}
The first two cases of Eq.~\eqref{e:cat-count} follow trivially. For the third case, $\Phi_d(v_f)$ depends on the way $d$ sits inside $P$. Adding diagonal $d$ divides an $(s+t-2)$-sided subpolygon of $P$ into an $s$-gon and a $t$-gon.   Since the Catalan number $C_{n-2}$ enumerates the triangulations of the $n$-gon, the result follows.
\end{proof}

\begin{lem} \label{l:iff}
Let $S = \{f_0, f_1, \dots, f_m\}$ be a subflag of the associahedron $K_n$, and let $x$ be a point in $\Phi(\Delta_S)$ such that
$$x \ = \ a_0 \Phi(v_{f_0}) \ + \ a_1\Phi(v_{f_1}) \ + \ \dots \ + \ a_m\Phi(v_{f_m})\, ,$$
where $\sum a_i = 1$ and each $a_i \geq 0$. 
Then diagonal $d$ is in $f_k$ if and only if 
\begin{equation} 
\label{e:eq}
x_d \ = \ \sum_{i=0}^{k-1} \ a_i \ \Phi_d(v_{f_i}) \ + \ \frac{1}{n-3} \cdot \left( 1 - \ \sum_{i=0}^{k-1} a_i \right).
\end{equation}
\end{lem}

\begin{proof} 
If diagonal $d \in f_k$, then $d \in f_i$, for all $i \geq k$.  Proposition~\ref{p:cp} then implies that $\Phi_d(v_{f_i})=1/(n-3)$, for all $i \geq k$, and the result follows.
Conversely, Eq.~\eqref{e:eq} yields
$$\sum_{i=0}^{k-1} \ a_i \ \Phi_d(v_{f_i})  \ = \  \frac{1}{n-3} \cdot \sum_{i=k}^{m} \ a_i \, .$$
Since each $\Phi_d(v) \leq 1/(n-3)$, it follows by Proposition~\ref{p:cp} that $\Phi_d(v_{f_i}) = 1/(n-3)$, for all $i \geq k$, and the result follows.
\end{proof}

\begin{lem} 
\label{l:newlem}
Let $x$ be a point in the image of $\Phi$. Then there exists a unique subflag $\{ f_0, f_1, \dots f_m\}$ and coefficients $a_0, a_1, \dots a_m$, where $a_i> 0$ and $\sum a_i = 1$, such that 
$$x \ = \ a_0 \Phi(v_{f_0}) \ + \ a_1\Phi(v_{f_1}) \ + \ \dots \ + \ a_m\Phi(v_{f_m}).$$
\end{lem}

\begin{proof} 
We present an algorithm that uniquely determines $a_i$ and $f_i$ from the coordinates of $x$ through an iterative procedure. 
As the base case, by Lemma \ref{l:iff}, $f_0$ is the set of diagonals $d$ such that $x_d=1/(n-3)$. 
The procedure assumes that for some $j$, we know $f_0, f_1, \dots, f_j$ and $a_0, a_1, \dots, a_{j-1}$, and determines $a_j$ by either
\begin{enumerate}[(A)]
\item concluding $f_j$ is the maximal set of diagonals in the subflag, \emph{or}
\item identifying the next element $f_{j+1}$ in the subflag.
\end{enumerate}
In each iteration, every subpolygon\footnote{We say $S$ is a subpolygon of a polygonal representation $P$ if the boundary of $S$ consists of edges and diagonals of $P$ and no diagonals of the polygonal representation are drawn in the interior of $S$.} of the polygonal representation of $f_j$ is tested to obtain information that is used to determine $a_j$ and $f_{j+1}$, if applicable. 

Let $P$ be the polygonal representation of $f_j$, and $P_r$ be a subpolygon of $P$ with $r$ sides.  Let $D$ the set of diagonals in $P_r$ that partition it into a 3-gon and an $(r-1)-$gon. The following relates diagonals $d$ that can be drawn within $P_r$ to  coordinate values $x_d$.

\begin{claim}
Diagonals in $P_r$ are not in any set of diagonals of the subflag if and only if 
\begin{equation}
\label{e:star}  
x_d \ =  \ \sum_{i=0}^{j-1} \ a_i \ \Phi_d(v_{f_i}) \ + \ \frac{1}{n-3} \cdot \frac{C_{r-1}}{C_r} \cdot \left( 1 - \ \sum_{i=0}^{j-1} a_i \right).
\end{equation}
for all diagonals $d$ of $D$: 
\end{claim}

\begin{proof}
If no diagonal in $P_r$ is contained in any diagonals of the subflag, then $P$ contains the region $P_r$, for all $i\geq j$. Therefore, by Proposition~\ref{p:cp}, 
$$\Phi_d(v_{f_i}) \ = \ \frac{1}{n-3} \cdot \frac{C_{r-1}}{C_r} \, ,$$
for $d$ in $D$ and $i \geq j$, and Eq.~\eqref{e:star} follows. Conversely, suppose there is a diagonal $d_\ast$ in $P_r$ that is contained in diagonals of the subflag. Let $f_k$ be the face containing the minimal such set of diagonals. For a diagonal $d$ in $D$ that intersects $d_\ast$, $\Phi_{d}(v_{f_i}) = 0$  for all $i \geq k$  by Proposition~\ref{p:cp}.  Similarly, since $P_r$ is in face $f_i$ for all $j\leq i\leq k-1$, then
\begin{equation}
\label{e:minisum}
\Phi_{d}(v_{f_i}) \ = \ \frac{1}{n-3} \cdot \frac{C_{r-1}}{C_r} \, .
\end{equation} 
But because 
\begin{align*}
x_{d}& 
\ = \ \sum_{i=0}^{j-1} a_i \ \Phi_{d}(p_i)
\ + \ \frac{1}{n-3} \cdot \frac{C_{r-1}}{C_r} \cdot \sum_{i=j}^{k-1} a_i 
\ + \ 0 \cdot \sum_{i=k}^{m} a_i \\
&
\ < \ \sum_{i=0}^{j-1} a_i \ \Phi_{d}(p_i)
\ + \ \frac{1}{n-3} \cdot \frac{C_{r-1}}{C_r} \cdot \sum_{i=j}^{m} a_i\, ,
\end{align*}
Eq.~\eqref{e:star} does not hold for $d$, justifying the claim.
\end{proof}

If for \emph{every} subpolygon of $P$, diagonals in $P_r$ are not in any set of diagonals of the subflag, then $f_j$ is the maximal split system in the subflag.  It follows that $a_j = 1-a_0-a_1-\dots-a_{j-1}$, resulting in case (A) and ending the algorithm.
Otherwise, there exists a diagonal in some subpolygon $P_r$ of $f_j$ contained in some set of diagonals of the subflag, resulting in case (B).  For each such subpolygon, we compute  $a_j+a_{j+1}+ \dots + a_{k-1}$, where $k$ is such that $f_k$ is the minimal set of diagonals containing a diagonal from $P_r$, and we use this information to determine the diagonals of $f_{j+1}$ and $a_j$. 

First we show how to compute $a_j+a_{j+1}+ \dots + a_{k-1}$, where $k$ is such that $f_k$ is the minimal set of diagonals containing a diagonal from $P_r$. Let  $d_\ast$ be a diagonal in $f_k$ and $P_r$. For all diagonals $d$ in $T$, since $P_r$ is in $f_i$, for all $j\leq i\leq k-1$, Eq.~\eqref{e:minisum} holds, and it follows that
\begin{equation} 
\label{e:startwo}
x_d \ - \ \sum_{i=0}^{j-1} a_i \ \Phi_d(v_{f_i})  \ = \ \frac{1}{n-3} \cdot \frac{C_{r-1}}{C_r} \cdot \sum_{i=j}^{k-1} a_i  \ + \  \sum_{i=k}^{m} a_i \ \Phi_d(v_{f_i}).
\end{equation}  
For a diagonal $d$ in $T$ that crosses a diagonal of $f_k$, $\Phi_d(v_{f_i})=0$, for $i\geq k$, and so Eq.~\eqref{e:startwo} becomes 
$$x_d \ - \ \sum_{i=0}^{j-1} a_i \ \Phi_d(v_{f_i})  \ = \ \frac{1}{n-3} \cdot \frac{C_{r-1}}{C_r} \cdot \sum_{i=j}^{k-1} a_i.$$ 
If $d$ does not cross diagonals in $f_k$, then $\Phi(v_{f_k})>0$ and Eq.~\eqref{e:startwo} becomes
$$x_d \ - \ \sum_{i=0}^{j-1} a_i \ \Phi_d(v_{f_i})  \ > \ \frac{1}{n-3} \cdot \frac{C_{r-1}}{C_r} \cdot \sum_{i=j}^{k-1} a_i.$$ 
Since at least one diagonal $d$ in $T$ intersects $f_k$, 
$$\min \left\{ x_d \ -\ \sum_{i=0}^{j-1} a_i \ \Phi_d(v_{f_i}) \suchthat d \in T\ \right\} \ = \ \frac{1}{n-3} \cdot \frac{C_{r-1}}{C_r} \cdot \sum_{i=j}^{k-1} a_i\, ,$$
allowing us to solve for $a_{j}+ \dots + a_{k-1}$ by computing this minimum.

Now we use the values $a_{j}+ \dots + a_{k-1}$ for each subpolygon to determine $a_j$ and $f_{j+1}$. If $P_r$ contains a diagonal in $f_{j+1} \setminus f_j$, the value returned by the subroutine on $P_r$ must be $a_j$. But if $P_r$ does not contain such a diagonal, the returned value is at least $a_j + a_{j+1}$.  Thus, the minimum value returned over all subpolygons of $f_j$ is $a_j$. Since we know $a_0, a_1, \dots a_j$ and $f_0, f_1, \dots f_{j}$, use Lemma~\ref{l:iff} to test whether this statement holds for each diagonal $d$, and in doing so determine precisely which diagonals are in $f_{j+1}$.  

This process is repeated with the additional known values $a_j$ and $f_{j+1}$ until case (A) is reached, at which point the values of $a_i$ and $f_i$ are fully determined.  Since at no point was there freedom in choosing values for $a_i$ or $f_i$, these values are uniquely determined.
\end{proof}

The following shows $\Phi$ preserving the dimension of the simplicial subdivision. 

\begin{lem} \label{l:simplex}
If $F$ is a flag of the associahedron $K_{n-1}$, then $\Phi(\Delta_F)$ is an $(n-3)$-simplex. More generally, if $S$ is a subflag of $K_n$, then $\Phi(\Delta_S)$ is an $(|S|-1)$-simplex. 
\end{lem}

\begin{proof} 
Let $F=\{f_0, f_1, \dots f_{n-3}\}$ be a flag of the associahedron $K_{n-1}$ and let $S_i =\{ f_0, f_1, \dots f_i\}$. Proceed by induction on $i$. Since $S_0$ has cardinality one, $\Phi(v_{f_0})$ is a point in $\n{n}$.  Now assume $\Phi(\Delta_{S_i})$ is an $i$-simplex, the convex hull of $\Phi(v_{f_0}),  \dots \Phi(v_{f_{i}})$. In order to show $\Phi(\Delta_{S_{i+1}})$ is an $(i+1)$-simplex, it suffices to show that $\Phi(v_{f_{i+1}})$ is not in the convex hull of $\Phi(v_{f_0}),  \dots \Phi(v_{f_{i}})$. 
Let $d$ be the diagonal in $f_{i+1}$ that is not in $f_i$. By Lemma \ref{l:iff}, $\Phi_d(v_{f_{i+1}})=1/(n-3)$ and $\Phi_d(v_{f_j}) < 1/(n-3)$, for $j<i+1$, demonstrating the claim.

And since $\Phi(\Delta_F)$ is an $(n-3)$-simplex, arising as the convex hull of the $n-2$ points $\Phi(v_{F_0}),  \dots \Phi(v_{F_{n-3}})$, the convex hull of any subset of these points (coming from a subflag) will form a simplex with dimension one less than the cardinality of the subset.
\end{proof}

%%%%%%%%%%%%%%%%%%%%%%%%%%%%%%%%%%%%%%%%%%%%%%%%%%%%%%%%%%%%%%%%%%%%%%%%%%%%%%%%%%%%%%%%%
%
%                  REFERENCES
%
%%%%%%%%%%%%%%%%%%%%%%%%%%%%%%%%%%%%%%%%%%%%%%%%%%%%%%%%%%%%%%%%%%%%%%%%%%%%%%%%%%%%%%%%%
\bibliographystyle{amsplain}

\end{document}